\documentclass[preprint,12pt]{elsarticle}

\usepackage{amssymb}%
\usepackage{graphicx}%
\usepackage{multirow}%
\usepackage{amsmath,amssymb,amsfonts}%
\usepackage{amsthm}%
\usepackage{mathrsfs}%
\usepackage[title]{appendix}%
\usepackage{xcolor}%
\usepackage{textcomp}%
\usepackage{manyfoot}%
\usepackage{booktabs}%
\usepackage{algorithm}%
\usepackage{algorithmicx}%
\usepackage{algpseudocode}%
\usepackage{enumitem}%
\usepackage{listings}%
\usepackage{hyperref}
\usepackage[capitalise]{cleveref}%
\usepackage{orcidlink}%
\usepackage{microtype}%
\usepackage[size=tiny]{todonotes}
\usepackage[T1]{fontenc}
\usepackage{lmodern}

\newtheorem{theorem}{Theorem}%
\newtheorem{definition}{Definition}%
\newtheorem{property}{Property}%
\newtheorem{corollary}{Corollary}%
\newtheorem{lemma}{Lemma}%
\newtheorem{observation}{Observation}%

\newcommand{\intInterval}[1]{[\![#1]\!]}

\journal{Discrete Applied Mathematics}

\begin{document}
\begin{frontmatter}

	\title{Finding \texorpdfstring{$k$}{k}-community structures in special graph classes}

	\author[1]{Narmina Baghirova\corref{cor1}\orcidlink{0000-0002-4702-673X}}
	\cortext[cor1]{Corresponding author}
	\ead{narmina.baghirova@unifr.ch}
	\author[2]{Cl{\'e}ment Dallard\orcidlink{0000-0002-9522-3770}}
	\ead{clement.dallard@ens-lyon.fr}
	\author[1]{Bernard Ries\orcidlink{0000-0003-4395-5547}}
	\ead{bernard.ries@unifr.ch}
	\author[1]{David Schindl\orcidlink{0000-0002-7009-5530}}
	\ead{david.schindl@unifr.ch}

	\affiliation[1]{organization={Department of Informatics, University of Fribourg},%
		addressline={Bd de Perolles 90},
		city={Fribourg},
		postcode={1700},
		country={Switzerland}}

	\affiliation[2]{organization={LIP, ENS de Lyon},%
	addressline={46 All{\'e}e d'Italie},
	city={Lyon},
	postcode={69364},
	country={France}}

	\begin{abstract}
		For an integer $k\ge 2$, a \emph{$k$-community structure} in an undirected graph is a partition of its vertex set into $k$ sets called \textit{communities}, each of size at least two, such that every vertex of the graph has proportionally at least as many neighbours in its own community as in any other community.
		In this paper, we give a necessary and sufficient condition for a forest on $n$ vertices to admit a $k$-community structure.
		Furthermore, we provide an $\mathcal{O}(k^2\cdot n^{2})$-time algorithm that computes such a $k$-community structure in a forest, if it exists.
		These results extend a result of [Bazgan et al., Structural and algorithmic properties of $2$-community structure, Algorithmica, 80(6):1890-1908, 2018].
		We also show that if communities are allowed to have size one, then every forest with $n \geq k\geq 2$ vertices admits a $k$-community structure that can be found in time $\mathcal{O}(k^2 \cdot n^{2})$.
		We then consider threshold graphs and show that every connected threshold graph admits a $2$-community structure if and only if it is not isomorphic to a star; also if such a $2$-community structure exists, we explain how to obtain it in linear time.
		We further describe an infinite family of disconnected threshold graphs, containing exactly one isolated vertex, that do not admit any $2$-community structure.
		Finally, we present a new infinite family of connected graphs that may contain an even or an odd number of vertices without $2$-community structures, even if communities are allowed to have size one.
	\end{abstract}

	\begin{keyword}
		graph partitioning \sep algorithms \sep complexity \sep community structure \sep threshold graphs \sep forests
	\end{keyword}
\end{frontmatter}

\section{Introduction}\label{sec-intro}
When it comes to modelling social networks, graphs are a powerful tool.
The vertices of the graphs often represent individuals, while the edges represent links/relationships between them.
This underlying model allows us to analyse social networks from a structural perspective.
In the research field around social networks, a particular focus has been put on detecting so-called \textit{communities}.
In a network, a community can intuitively be seen as a subset of vertices of the graph that are more densely connected to each other than to the vertices of the rest of the network.
Problems motivated by community detection in networks can mostly be put under the same umbrella as vertex partitioning problems and problems related to finding dense subgraphs. 
These problems have been widely studied and have applications in many different domains like for instance numerical analysis~\cite{GPNumAnylysis}, scientific simulations~\cite{GPScientificSim}, bioinformatics~\cite{GPBioinformatics} and quantum computing~\cite{GPQuantumComp} (see~\cite{graphpartitioning} for an overview regarding partitioning problems).

In partitioning problems, the parts of a valid partition must respect some constraints expressed in terms of other parameters, such as the number of edges between the parts, the number of neighbours of each vertex within the different parts, or the size of the parts.
However, these constraints usually do not combine those parameters.

For instance, let us consider the \emph{Satisfactory Partition} problem introduced in~\cite{SatPartition}.
This problem and several of its variants have been intensively studied (see for instance~\cite{NoteOnSatPart,BalSatPart,SatPartParamComp,SatPartAlgApproach}).
The \emph{Satisfactory Partition} problem consists in deciding if a given graph has a partition of its vertex set into two nonempty parts such that each vertex has at least as many neighbours in its own part as in the other part.
This condition may not reflect the \textit{natural} partition into two communities that the graph could admit. Consider the example in which a vertex/individual is linked to 10 other vertices/individuals belonging to community~1 and to 100 vertices/individuals belonging to community~2. According to the definition of a satisfactory partition, this vertex/individual clearly belongs to community~2. However, if there is no further vertex/individual in the network belonging to community~1, but there are 900 further vertices/individuals in the network belonging to community~2 (but not linked to our particular vertex/individual), then our vertex/individual is connected to everyone in community~1, while only linked to a small proportion of vertices/individuals in community~2. Thus, it seems natural that this vertex/individual should be a part of community~1.

Therefore, instead of considering the exact number of neighbours in each part, one may focus on the \emph{proportion} of neighbours in each part and require that each vertex has proportionally at least as many neighbours in its own part as in the other part.
This new constraint effectively combines two parameters: the number of neighbours within the parts and the sizes of the parts.
This was the standpoint suggested by Olsen in 2013 (see~\cite{olsenGenView}), who argued that the notion of \emph{proportionality} in the definition of community structures is both intuitive and supported by observations in real world networks; something that previous attempts at defining communities (see for example~\cite{def2,def3}) failed to capture.

In this paper, we follow Olsen's view on communities.
The formal definition of a community structure is as follows.

\begin{definition}(Olsen~\cite{olsenGenView})
	A \emph{community structure} of a connected graph $G=(V,E)$ is a partition $\Pi$ of the vertex set $V$ such that:
	\begin{itemize}
		\item $|\Pi|\ge 2$,
		\item for every $C\in \Pi$, $\vert C\vert \ge 2$,
		\item for every $C\in \Pi$, every $v \in C$ and every $C'\in \Pi$, $C\neq C'$, the following property holds:
		      \begin{equation}
			      \frac{\vert N_{C} (v)\vert}{\vert C \vert -1} \ge \frac{\vert N_{C'} (v)\vert}{\vert C' \vert}
		      \end{equation}
		      where $N_{C} (v)$ (resp.\ $N_{C'} (v)$) is the set of neighbours of $v$ in $C$ (resp.\ in $C'$).
	\end{itemize}
\end{definition}

In other words, a community structure of a connected graph $G=(V,E)$ is a partition of its vertex set into at least two sets (called the \textit{communities} of $G$), each containing at least two vertices, such that every vertex is adjacent to proportionally at least as many vertices in its own community as to vertices in any other community.

In~\cite{olsenGenView}, Olsen showed that a community structure can be found in polynomial time in any connected graph containing at least four vertices, except the stars.
On the other hand, he showed that determining whether a graph admits a community that contains a predefined set of vertices is \textsf{NP}-complete.
Notice that in the definition introduced by Olsen in~\cite{olsenGenView}, the exact number of communities is not given, i.e., the only restriction is that there are at least two communities.
In~\cite{bazgan:k-comm, Estivill1}, the notion of \textit{$k$-community structure} was first used in order to fix the number of communities to some integer $k\geq 2$, i.e., a $k$-community structure is a community structure $\Pi$ with $|\Pi|=k$, meaning it contains exactly $k$ communities.

In~\cite{Estivill2}, it was also shown that deciding whether there exists a connected $k$-community structure (i.e. a $k$-community structure in which the vertices of each community induce a connected subgraph) such that each community has the same size is \textsf{NP}-complete. Note that $k$ is, in this case, part of the input.
However, much more results have been obtained for the special case when $k=2$.
First, it was shown in~\cite{Estivill1}, that deciding whether a graph admits a $2$-community structure such that both communities have equal size is \textsf{NP}-complete.
In~\cite{bazgan:k-comm,Estivill2}, the authors showed independently that every tree with $n\ge4$ vertices which is not isomorphic to a star admits a $2$-community structure and that such a $2$-community structure can be found in time $\mathcal{O}(n)$, even if we require the vertices in each community to induce a connected subgraph.
Furthermore, the authors in \cite{bazgan:k-comm} also showed that except for stars, graphs of maximum degree 3, graphs of minimum degree at least $\vert V \vert-3$, and complement of bipartite graphs always admit a $2$-community structure, which can be found in polynomial time, even if we require connectivity of the communities.
Recently, the authors of~\cite{LCP} introduced a framework, which allows to solve a special family of partitioning problems in polynomial time in classes of graphs of bounded clique-width. As an application, they showed that the problem of deciding whether there exists a $k$-community structure and finding such a structure, if it exists, can be solved in polynomial time in classes of graphs of bounded clique-width. 
Notice that in general graphs, the complexity of deciding whether a given graph admits a $k$-community structure is still open, even if $k$ is a fixed integer (not part of the input).
In fact, until 2020 no graphs (except for stars) not admitting any $2$-community structure were known; the first infinite family of such graphs was presented in~\cite{bazgan:firstInfiniteFamily}.

In the original definition of community structure introduced by Olsen, and the one of $k$-community structure introduced in~\cite{bazgan:k-comm, Estivill1}, each community must contain at least two vertices.
One may relax this constraint and only require a community to be non-empty, i.e., to contain at least one vertex.
It is important to note that allowing communities of size one does not make the problem necessarily trivial and is a natural generalization of the above definition of a community structure.
In our paper, we use this version as well and call such a partition a \textit{generalized $k$-community structure}.
In~\cite{bazgan:firstInfiniteFamily}, the authors introduce the notion of \textit{proportionally dense subgraph (PDS)}, and in particular so-called \textit{$2$-PDS partitions}, which correspond exactly to generalized $2$-community structures.
Our notion of generalized $k$-community structure can therefore be seen as a generalisation of $2$-PDS partitions to $k$-PDS partitions.
The authors of~\cite{bazgan:firstInfiniteFamily} present two infinite families of graphs: (i)~one infinite family containing graphs with an even number of vertices that do not contain any generalized $2$-community structure; (ii)~one infinite family containing graphs that do not admit any connected generalized $2$-community structure.

In this paper, we consider several graph classes and contribute with the following results.
First, we extend a result of~\cite{bazgan:k-comm,Estivill2} by characterizing, for any $k\geq 2$, the forests on $n$ vertices admitting a $k$-community structure. We also propose an $\mathcal{O}(k^2\cdot n^{2})$ algorithm to construct such a $k$-community structure, if it exists (see \cref{sec:trees}), and provide a similar result for generalized $k$-community structures. Notice that a polynomial time algorithm for finding a $k$-community structure in forests was already provided in~\cite{LCP}, since such graphs have clique-width at most 3. However, the proposed algorithm is XP parametrized by $k$. We therefore substantially improve on its complexity.
Second, we show that any threshold graph admits a generalized 2-community structure that can be found in linear time.
Then, we give a characterization of connected threshold graphs on $n$ vertices and $m$ edges that admit a $2$-community structure, and show that, if it exists, it can be found in time $\mathcal{O}(n + m)$.
For the case of disconnected threshold graphs, we show that the only such graphs not admitting a $2$-community structure must contain exactly one isolated vertex.
We exhibit an infinite family of such graphs (see \cref{sec-thresh}).
Finally, we present an infinite family of connected graphs not admitting any generalized $2$-community structure.
In contrast to~\cite{bazgan:firstInfiniteFamily}, where another such family has been presented but with the restriction that every graph of the family has an even number of vertices, our family contains graphs with an even number of vertices and graphs with an odd number of vertices (see \cref{sec-family}).

\section{Preliminaries}
\label{sec-prelim}

\subsection{Definitions and notation}

In this paper, all graphs are simple and undirected.
Let $G=(V,E)$ be a graph.
The \emph{neighbourhood} of a vertex $v\in V$ is denoted by $N(v)$, its \emph{closed neighbourhood} by $N[v] := N(v) \cup \{v\}$ and its \emph{degree} by $d(v):=\vert N(v)\vert$.
The \emph{neighbourhood} of $v\in V$ \emph{in $V'\subseteq V$}, denoted by $N_{V'}(v)$, is the set of vertices that are both in $V'$ and adjacent to $v$, i.e., $N_{V'}(v):=N(v)\cap V'$.
We say that a vertex $v\in V$ is \emph{universal} if $N[v]=V$.
For $v_1,v_2 \in V$, we say that $v_1$ and $v_2$ are \emph{true twins}, if $N[v_1]=N[v_2]$; we say that they are \emph{false twins}, if $N(v_1)=N(v_2)$.

A \emph{$k$-partition $\Pi$ of $G$} is a partition of $V$ into $k$ subsets $C_1,\ldots,C_k$.
The \textit{subgraph} of $G=(V,E)$ \textit{induced by $S\subseteq V$} is defined as $G[S]:=(S,\{uv \in E: u\in S \text{ and } v\in S\})$.
We denote by $\intInterval{t}$, with $t \in \mathbb{N}$, the set $\{1, 2, \ldots, t\}$.
We say that a $k$-partition $\Pi=\{C_1,\ldots,C_k\}$ is \emph{connected}, if $G[C_i]$ is a connected subgraph, for all $i \in\intInterval{k}$.

A \emph{tree} is a connected, acyclic graph and a \emph{forest} is an acyclic graph.
A \emph{star} $S_n$ is a tree on $n+1$ vertices, where exactly one vertex has degree $n$ (called the \textit{center}) and all of its $n$ neighbours have degree $1$ (called the \textit{leaves}).

Let $A$ be a finite, totally ordered set, called \emph{alphabet}.
The \emph{lexicographic order} on the set of all $k$-tuples of symbols from $A$, for some positive integer $k$, is the total order such that, for two such distinct tuples $a = (a_1,a_2,\ldots, a_k)$ and $b = (b_1,b_2,\ldots, b_k)$, the tuple $a$ is smaller than $b$ with respect to the lexicographic order, denoted by $a < b$, if and only if there exists $i \in \intInterval{k}$ with $a_i < b_i$ and $a_j=b_j$ for all $j<i$ in the underlying order of the alphabet~$A$.

Let us now formally define the main concept of our paper, namely $k$-community structures (as used in~\cite{bazgan:k-comm}).

\begin{definition}
	A (connected) \emph{$k$-community structure} of a graph $G=(V,E)$ is a (connected) $k$-partition $\Pi = \{C_1,\ldots,C_k\}$ of the vertex set $V$ such that:
	\begin{itemize}
		\item $k\ge 2$,
		\item for all $i\in \intInterval{k}$, $\vert C_i \vert \ge 2$,
		\item for all $i \in \intInterval{k}$, all $v \in C_i$ and all $j\in \intInterval{k}$, $j\neq i$, the following property holds:
		      \begin{equation}
			      \label{prop:k-comm2}
			      \frac{\vert N_{C_i} (v)\vert}{\vert C_i \vert -1} \ge \frac{\vert N_{C_j} (v)\vert}{\vert C_j \vert}\,.
		      \end{equation}
	\end{itemize}
\end{definition}

We say that a vertex $v \in C_i$, for $i \in \intInterval{k}$, is \emph{satisfied with respect to $\Pi$}, if it satisfies \eqref{prop:k-comm2} for all $j\in \intInterval{k}\backslash \{i\}$.
In this paper, we sometimes allow communities to have size one, thus slightly generalize the definition of a $k$-community structure.
In the above definition, besides changing $\vert C_i \vert \ge 2$ into $\vert C_i \vert \ge 1$, we reformulate \eqref{prop:k-comm2} as
\begin{equation}
	\label{prop:k-comm1}
	\vert N_{C_i} (v)\vert \cdot \vert C_j \vert \ge \vert N_{C_j} (v)\vert \cdot(\vert C_i \vert -1)\,.
\end{equation}

We call such a $k$-partition $\Pi$ a \emph{generalized $k$-community structure}, and we say that a vertex $v \in C_i$, for $i \in \intInterval{k}$, is \emph{satisfied with respect to $\Pi$}, if it satisfies \eqref{prop:k-comm1} for all $j\in \intInterval{k}\backslash \{i\}$.

Notice that the non-existence of a generalized $k$-community structure in a graph implies the non-existence of a $k$-community structure in this graph, but the converse is not necessarily true.
Indeed, consider the star $S_{n}$ with center $u$ and $n$ leaves $v_1,\ldots,v_n$.
As already observed in~\cite{olsenGenView}, $S_n$ does not admit any $k$-community structure for $k\geq 2$, but it contains a generalized $k$-community structure for any $2\leq k\leq n+1$.
Indeed, for $i=1,\ldots,k-1$, set $C_i=\{v_i\}$, and finally set $C_k=\{u,v_k,\ldots,v_n\}$.
It is easy to see that $\{C_1,\ldots,C_k\}$ yields a generalized $k$-community structure.
Further, notice that any property that holds for a generalized $k$-community structure holds in particular for a $k$-community structure, but the converse is not necessarily true.

\subsection{Properties}
\label{sec:properties}

In this section, we introduce several properties that are useful to show our results in the upcoming sections.
First, we introduce two properties that give sufficient conditions for a vertex to belong to a particular community within a (generalized) community structure.

\begin{property}
	\label{allneighboursSameCommunity}
	Let $G=(V,E)$ be a graph, let $\Pi=\{C_1, C_2\}$ be a $2$-community structure of $G$ and let $v \in V$ such that $d(v)\geq 1$.
	If $N_{C_i}(v)=\emptyset$, for some $i\in \{1,2\}$ then $v\in C_{3-i}$.
\end{property}

\begin{proof}
	Since $d(v)\geq 1$, $N_{C_i}(v)=\emptyset$ implies  $N_{C_{3-i}}(v)\ge 1$.
	Assume by contradiction that $v\in C_i$.
	Then, inequality \eqref{prop:k-comm1} applied to $v$ yields the following contradiction:
	\[\vert N_{C_i} (v)\vert \cdot \vert C_{3-i} \vert=0  \ge \vert N_{C_{3-i}}(v)\vert \cdot (\vert C_i \vert -1)\geq 1\,.\]
\end{proof}

\begin{property}
	\label{allCommunityInneighbourhood}
	Let $G=(V,E)$ be a graph and $\Pi=\{C_1, \ldots, C_k\}$ be a generalized $k$-community structure of~$G$.
	Consider $i, j \in \intInterval{k}$, $i\neq j$ and $v \in V$.
	If $C_i\subseteq N[v]$ and $C_j\not\subseteq N[v]$, then $v\not\in C_j$.
\end{property}

\begin{proof}
	Assume that $v\in C_j$.
	Then we have
	\[\vert N_{C_j} (v) \vert \cdot \vert C_i \vert < (\vert C_j\vert - 1)\cdot \vert C_i \vert  =  (\vert C_j\vert - 1) \cdot \vert N_{C_i}(v)\vert ,\]
	which also contradicts \eqref{prop:k-comm1}.
\end{proof}

If we have a community of size one, we obtain the following corollary.

\begin{corollary}
	\label{kcommUnivSizeOne}
	Let $G=(V,E)$  be a graph and let $\Pi=\{C_1, \ldots, C_k\}$ be a generalized $k$-community structure of~$G$.
	If $C_i=\{u\}$ for some $i\in \intInterval{k}$ and $u\in V$, then for every neighbour $v$ of $u$, we have $C_j \subseteq N[v]$, where $v \in C_j, j\neq i$.%
\end{corollary}

In particular, if $k=2$, the condition is also sufficient.
Therefore, the next property characterizes those graphs admitting a generalized $2$-community structure, where at least one community is of size 1.

\begin{property}
	\label{2commUniv}
	A graph $G=(V,E)$ admits a generalized $2$-community structure $\{C_1,C_2\}$ such that $C_i=\{u\}$, for some $i\in \intInterval{2}$ and $u\in V$, if and only if every neighbour of $u$ is universal.
\end{property}
\begin{proof}
	The necessity follows from \Cref{kcommUnivSizeOne} with $k=2$, and the sufficiency follows from the observation that $\{\{u\},V\backslash \{u\}\}$ is a generalized $2$-community structure of~$G$.
\end{proof}

The next property is used to show that a graph $G=(V,E)$ does not admit a certain generalized $2$-community structure $\{C_1,C_2\}$.

\begin{property}\label{prop:preassignment}
	Let $G=(V,E)$ be a graph and let $C_1', C_2'\subset V$ such that $C_1',C_2'\neq \emptyset$ and $C_1'\cap C_2' =\emptyset$.
	Let $U=V\backslash (C_1'\cup C_2')$ and $v\in U$.
	If $v$ is not satisfied with respect to $\{C_1'\cup (U\cap N[v]), C_2'\cup (U\backslash N[v])\}$, then $G$ admits no generalized $2$-community structure $\Pi=\{C_1,C_2\}$ such that $(C_1'\cup\{v\})\subseteq C_1$ and $C_2'\subseteq C_2$.
\end{property}

\begin{proof}
	Notice that by assumption $|C_1'|,|C_2'|\geq 1$ and hence, $|C_1|\geq 2$.
	If $v$ is not satisfied with respect to $\{C_1'\cup (U\cap N[v]), C_2'\cup (U\backslash N[v])\}$, we know that
	\begin{equation}
		\frac{\vert N_{C_1'\cup U} (v)\vert}{\vert C_1'\cup (U\cap N[v]) \vert-1} <
		\frac{\vert N_{C_2'} (v)\vert}{\vert C_2'\cup (U\backslash N[v]) \vert}\,.\label{vNotSatisfied}
	\end{equation}
	Notice, that since $v \in U = V\backslash (C_1'\cup C_2')$ and $|C_1'|\geq 1$, we have that $\vert C_1'\cup (U\cap N[v]) \vert> 1$.
	Let $\Pi=\{C_1,C_2\}$ be a partition of $V$ such that $(C_1'\cup\{v\})\subseteq C_1$ and $C_2'\subseteq C_2$. %
	Since $|N_{C_1}(v)|\leq |C_1|-1$,  we have
	\begin{equation}
		\label{equation2}
		\frac{|N_{C_1}(v)|}{|C_1|-1}\leq \frac{|N_{C_1}(v)| + |N_{U\backslash C_1}(v)|}{|C_1|-1+|(N(v)\cap U)\backslash C_1|}= \frac{|N_{C_1\cup U}(v)|}{|C_1\cup (N(v)\cap U)|-1}\,.
	\end{equation}
	Now, since $C_1\subseteq V\backslash C_2' =C_1'\cup U$, we have $C_1\cup U=C_1'\cup U$, and \eqref{equation2} can be rewritten as
	\begin{equation}
		\frac{|N_{C_1}(v)|}{|C_1|-1}\leq\frac{|N_{C_1'\cup U}(v)|}{|C_1\cup (N(v)\cap U)|-1}\,.
	\end{equation}
	From the fact that $(C_1'\cup \{v\})\subseteq C_1$ and from \eqref{vNotSatisfied} it then follows that
	\begin{equation}
		\label{equation4}
		\frac{|N_{C_1}(v)|}{|C_1|-1}\leq\frac{|N_{C_1'\cup U}(v)|}{|C_1'\cup (N[v]\cap U)|-1}<\frac{\vert N_{C_2'} (v)\vert}{\vert C_2'\cup (U\backslash N[v]) \vert}\,.
	\end{equation}
	Since $C_2\subseteq C_2'\cup U$ and $v\not\in C_2$, \eqref{equation4} becomes
	\begin{equation}
		\frac{|N_{C_1}(v)|}{|C_1|-1}<\frac{\vert N_{C_2'} (v)\vert}{\vert C_2'\cup (C_2\backslash N[v]) \vert} = \frac{\vert N_{C_2'} (v)\vert}{\vert C_2' \cup (C2\backslash N(v))\vert}\,.
	\end{equation}
	Also, since $\vert N_{C_2'} (v)\vert\leq \vert C_2'\cup (C_2\backslash N(v)) \vert$  we have
	\begin{equation}
		\frac{|N_{C_1}(v)|}{|C_1|-1}<\frac{\vert N_{C_2'} (v)\vert+\vert N_{C_2\backslash C_2'} (v)\vert}{\vert C_2'\cup (C_2\backslash N(v)) \vert+\vert (N(v)\cap C_2)\backslash C_2'\vert}=\frac{\vert N_{C_2} (v)\vert}{\vert C_2\vert}\,.
	\end{equation}
	Thus, $v$ is not satisfied with respect to $\Pi$, and so $\Pi$ is not a generalized $2$-community structure.
\end{proof}

As mentioned above, we can use \Cref{prop:preassignment} to show that a graph $G=(V,E)$ does not admit a certain generalized $2$-community structure $\{C_1,C_2\}$.
Indeed, if some vertices have already been assigned to some sets $C_1'$ and $C_2'$, and if for some unassigned vertex $v \in U = V \setminus \{C_1',C_2'\}$, inequality \eqref{prop:k-comm1} fails even with respect to the partition $\{C_1'\cup (U\cap N[v]), C_2'\cup (U\backslash N[v])\}$, then in order to possibly complete $C_1',C_2'$ into a generalized $2$-community structure $(C_1,C_2)$ (i.e., $C_1'\subseteq C_1$ and $C_2'\subseteq C_2$), $v$ has to be assigned to~$C_2$. When we apply \Cref{prop:preassignment} and add $v$ and all its neighbours in $U$ to $C_i$, as well as all non-neighbours of $v$ in $U$ to $C_{3-i}$, for $i\in \{1,2\}$, we say that we \textit{test $v$ on $C_i$}.

Next, we consider true twins in generalized $k$-community structures.

\begin{property}
	\label{prop-false}
	Let $u$ and $v$ be two true twins in a graph $G=(V,E)$ that are not universal.
	Then, in any generalized $2$-community structure of $G$, $u$ and $v$ must belong to the same community.
\end{property}

\begin{proof}
	Consider a generalized $2$-community structure $\Pi=\{C_1,C_2\}$ of~$G$.
	Assume by contradiction that $u\in C_1$ and $v\in C_2$.
	Note that $|C_1|>1$, otherwise $C_1\subseteq N[v]$ and since $v$ is not universal, by \Cref{allCommunityInneighbourhood} it would not be satisfied with respect to~$\Pi$.
	Similarly, $|C_2|>1$.
	Thus, $\Pi$ must be a $2$-community structure.

	Since $u$ and $v$ are true twins, we have
	\begin{eqnarray}
		|N_{C_1}(u)|&=&|N_{C_1}(v)|-1\label{twinsC1}\,,\text{~and}\\
		|N_{C_2}(u)|&=&|N_{C_2}(v)|+1\label{twinsC2}\,.
	\end{eqnarray}

	Furthermore, since $C_1$ and $C_2$ form a $2$-community structure, we also have
	\begin{eqnarray}
		\frac{|N_{C_1}(u)|}{|C_1|-1}&\geq & \frac{|N_{C_2}(u)|}{|C_2|}\label{uwrtC2}\,,\text{~and}\\
		\frac{|N_{C_2}(v)|}{|C_2|-1}&\geq & \frac{|N_{C_1}(v)|}{|C_1|}\label{vwrtC1}\,.
	\end{eqnarray}

	Now, by using \eqref{twinsC1} and \eqref{twinsC2}, we can restate \eqref{uwrtC2} as

	\[\frac{|N_{C_1}(v)|-1}{|C_1|-1}\geq \frac{|N_{C_2}(v)|+1}{|C_2|}\,.\]

	On the other hand, $|N_{C_1}(v)|\leq |C_1|$ implies
	\begin{eqnarray}
		\frac{|N_{C_1}(v)|}{|C_1|}\geq \frac{|N_{C_1}(v)|-1}{|C_1|-1}\label{ineq1}
	\end{eqnarray}

	and $|N_{C_2}(v)|+1\leq |C_2|$ implies
	\begin{eqnarray}
		\frac{|N_{C_2}(v)|+1}{|C_2|}\geq \frac{|N_{C_2}(v)|}{|C_2|-1}\,.\label{ineq2}
	\end{eqnarray}

	Since $v$ is not universal, at least one of the inequalities \eqref{ineq1} and \eqref{ineq2} is strict.
	Since the vertices $u$ and $v$ are true twins and we can swap their roles, we may assume that \eqref{ineq1} is strict.
	Then, we have
	\[\frac{|N_{C_1}(v)|}{|C_1|}> \frac{|N_{C_1}(v)|-1}{|C_1|-1}\geq \frac{|N_{C_2}(v)|+1}{|C_2|}\geq \frac{|N_{C_2}(v)|}{|C_2|-1}\,,\]

	which contradicts \eqref{vwrtC1}.
\end{proof}

Our next property concerns $k$-community structures.
It gives a necessary condition for a graph (with no isolated vertices) to admit a $k$-community structure.

\begin{property}
	\label{k-matching}
	If a graph $G=(V,E)$, without isolated vertices, admits a $k$-community structure $\Pi$, then it has a matching of size~$k$.
\end{property}
\begin{proof}
	No community in $\Pi$ can contain a vertex with no neighbours in its own community.
	Then the vertex would have at least one neighbour in another community and could therefore not be satisfied with respect to~$\Pi$.
	Then, for each community $C_i \in \Pi$, $i\in \intInterval{k}$, there is at least one edge with both endpoints in~$C_i$.
	By picking one such edge in each community, we construct a matching of size $k$ in~$G$.
\end{proof}
Finally, we show that given an integer $k\ge 2$, a graph $G=(V,E)$ such that $|V|\ge k$ and a $k$-partition $\Pi$ of the vertex set of the graph, we can check in time $\mathcal{O}(k|V|+|E|)$ whether $\Pi$ forms a (connected) (generalized) $k$-community structure of~$G$.
\begin{lemma}\label{check if kcs linear time}
	Let $G=(V,E)$ be a graph such that $|V|\ge k$, for some given integer $k\ge 2$.
	Let $\Pi=\{C_1,\ldots,C_k\}$ be a $k$-partition of~$V$.
	Then, we can check in $\mathcal{O}(k(|V|+|E|))$ time whether $\Pi$ forms a (connected) (generalized) $k$-community structure of~$G$.
\end{lemma}
\begin{proof}
	We start by computing the size of each set $C_i$ in $\Pi$, for $i\in\{1,\ldots,k\}$.
	While doing so, we also mark which set each vertex belongs to and check if each set has at least two vertices.
	This can all be done in time $\mathcal{O}(|V|)$.
	Next, for each vertex we need to initialize $k$ counters and scan its neighbourhood to count the number of neighbours it has in each community.
	Since for each vertex this can be done in time $\mathcal{O}(k+d(v))$, overall this can be done in time $\mathcal{O}(k|V | + |E|)$.
	Then, for each vertex $v$ and each community $C_i$, for $i\in\{1,\ldots,k\}$, we check whether \eqref{prop:k-comm1} is satisfied by $v$ with respect to community~$C_i$.
	This can be done in time $\mathcal{O}(k|V|)$.
	If we want to check whether a (generalized) $k$-community structure is connected, we check whether each set in $\Pi$ induces a connected component by using either BFS or DFS.
	This can be done in time $\mathcal{O}(k(|V|+|E|))$.
	Hence, the claim follows.
\end{proof}

\section{(generalized) \texorpdfstring{$k$}{k}-community structures in trees and forests}
\label{sec:trees}

In this section, we prove that for every integer $k \geq 2$, every tree $T=(V,E)$ with at least $k$ vertices admits a connected generalized $k$-community structure, which can be found in time $\mathcal{O}(k^2 \cdot |V|^2)$.
Then, we give a necessary and sufficient condition for a tree $T=(V,E)$ to admit a $k$-community structure.
Furthermore, we show that a $k$-community structure of $T$, if it exists, can be found in time $\mathcal{O}(k^2 \cdot |V|^2)$.
We then generalize these results to forests.
Notice that in forests the connectivity of the communities is not necessarily preserved.
This extends a result from~\cite{bazgan:k-comm,Estivill2}, where the authors give a polynomial time algorithm for finding a $2$-community structure in trees.

We first introduce some definitions.
Let $k \geq 2$ be an integer and $\Pi = \{C_1,\ldots,C_k\}$ be a $k$-partition of a graph~$G$.
The \emph{size tuple} of $\Pi$ is the $k$-tuple $(s_1,\dots,s_k)$ whose elements correspond exactly to the elements from the multiset $\{|C_i|:C_i \in \Pi\}$ and such that $s_i \leq s_{i+1}$ for all $i \in \intInterval{k-1}$.
Note that $\{|C_i|:C_i \in \Pi\}$ is a multiset since we allow sets in $\Pi$ to have the same size.
A connected $k$-partition $\Pi$ of $G$ is said to be \emph{uniform} if its size tuple is lexicographically largest among all connected $k$-partitions of~$G$.
Let $S = (s_1, \dots, s_k)$ and $S' = (s_1', \dots, s_k')$ be two $k$-tuples.
We write $S < S'$ if there exists $j \in \intInterval{k}$ such that $s_j < s_j'$ and $s_i = s_i'$ for all $1 \leq i < j$.
For simplicity, for two $k$-partitions $\Pi$ and $\Pi'$ of $G$ with size tuples $S$ and $S'$, respectively, we write $\Pi < \Pi'$ whenever $S < S'$.

We start with a simple observation regarding trees and connected $k$-partitions.
Given a connected k-partition, if a vertex $v$ in a set has two neighbors $u$ and $w$ in another set $C$, then $v$ together with a path connecting $u$ and $w$ in $C$ (which exists since $C$ induces a connected subgraph) would induce a cycle.
Since a tree does not contain any cycle, we have the following.

\begin{observation}
	\label{observation1trees}
	Let $T=(V,E)$ be a tree, $k \geq 2$ an integer and $\Pi=\{C_1,\ldots,C_k\}$ a connected $k$-partition of~$T$.
	Let $v\in C_i$, for some $i\in \intInterval{k}$.
	Then $v$ has at most one neighbour in $C_j$, for any $j\in \intInterval{k} \setminus \{i\}$.
\end{observation}

We will now show that every tree $T=(V,E)$, such that $\vert V \vert \ge k$, admits a connected generalized $k$-community structure, which can be found in  time $\mathcal{O} (k^2\cdot|V|^2)$ (see \Cref{thrm: gen k-comm}).
We do so in two main steps.
First, we prove that if there exists a vertex in $V$ that does not satisfy \eqref{prop:k-comm1} with respect to some connected $k$-partition $\Pi$ of $T$, then we can update $\Pi$ in linear time into a connected partition $\Pi'$ of $T$ such that $\Pi < \Pi'$ (see \Cref{not sat implies not lex}).
Then, we present a polynomial upper bound on the number of times one can apply such an update, which results in a polynomial-time algorithm.

\begin{figure}[!ht]
	\centering
	\includegraphics[width=0.8\linewidth]{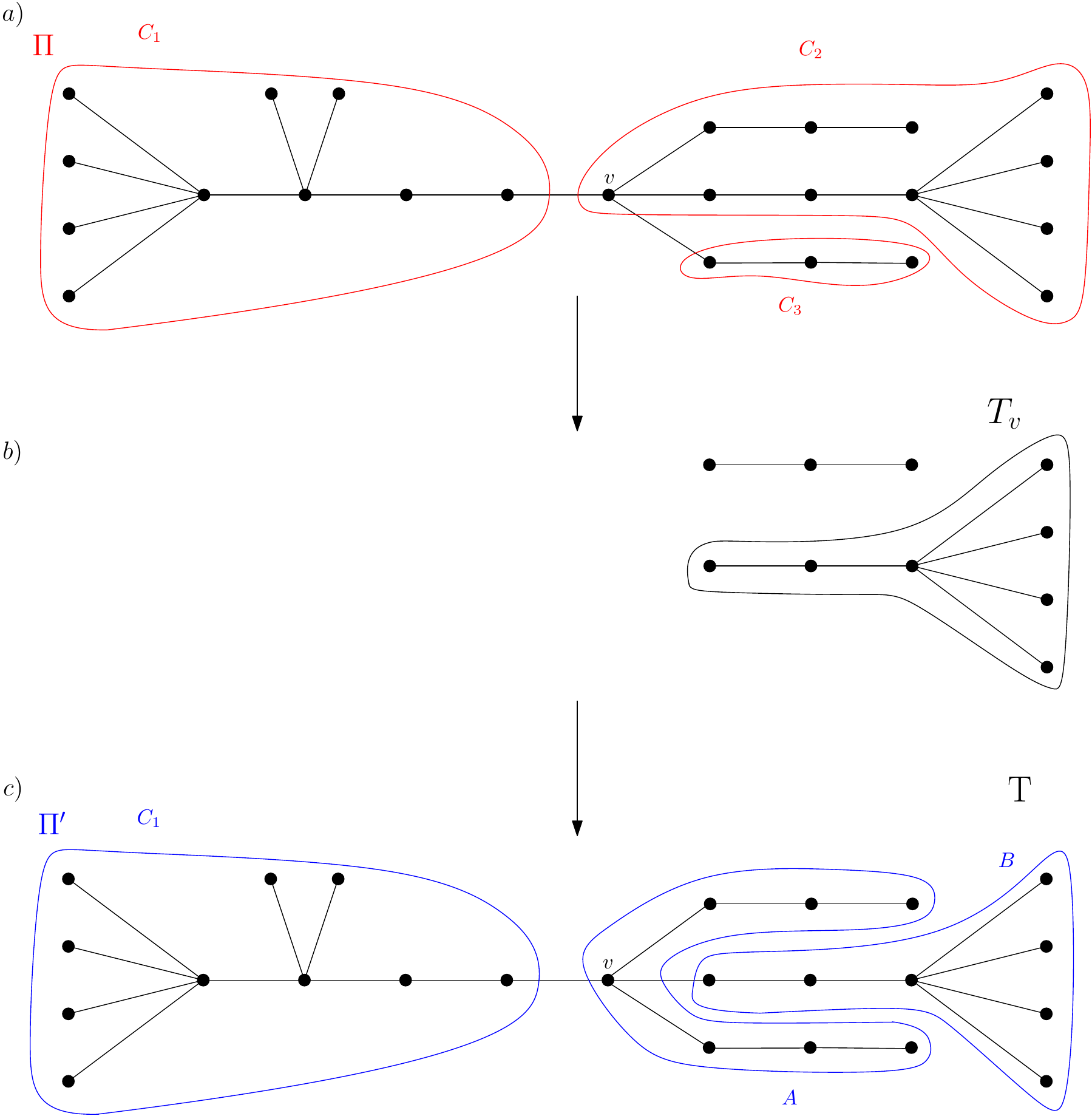}
	\caption{In part a) we illustrate a tree $T$ and a connected $3$-partition $\Pi$, where vertex $v \in C_2$ is not satisfied with respect to~$\Pi$.\newline
		In part b) we illustrate the forest induced by the vertices in $C_2\setminus \{v\}$ and a tree $T_v$ such that $\vert V(T_v)\vert > \vert C_3\vert$.\newline
		In part c) we illustrate the $2$-partition $\Pi' = (\Pi \setminus \{ C_2, C_3\}) \cup \{ A, B\}$.}
	\label{fig:treestransformation}
\end{figure}

\begin{lemma}\label{not sat implies not lex}
	Let $T=(V,E)$ be a tree such that $|V|\ge k$, for some integer $k\ge 2$.
	Let $\Pi$ be a connected $k$-partition of $T$ and let $v \in V$ be a vertex that is not satisfied with respect to~$\Pi$.
	Then there exists a connected $k$-partition $\Pi'$ of $T$ such that $\Pi < \Pi'$, that can be computed in time $\mathcal{O}(|V|)$.
\end{lemma}
\begin{proof}
	Let $\Pi=\{C_1,\ldots,C_k\}$ be a connected $k$-partition of~$T$.
	Assume that $v \in C_i$, for some $i \in \intInterval{k}$. %
	Then, by \Cref{observation1trees}, the following holds for $v$ and some $C_j$, for $i\neq j \in \intInterval{k}$:
	\begin{equation}
		\label{eq:P into Pi'}
		\vert C_i \setminus \{v\}\vert=\vert C_i \vert -1 \ge \vert N_{C_i}(v) \vert \cdot \vert C_j \vert +1 \,.
	\end{equation}
	Notice that $\vert N_{C_i}(v) \vert \ge 1$, since otherwise the inequality becomes $\vert C_i \vert \ge 2$, which in turn implies $\vert N_{C_i}(v) \vert \ge 1$ since the $k$-partition is connected.
	Therefore the \Cref{eq:P into Pi'} can be rewritten as
	$$\frac{\vert C_i \setminus \{v\}\vert}{\vert N_{C_i}(v)\vert} \ge \vert C_j \vert + \frac{1}{\vert N_{C_i}(v)\vert}> \vert C_j \vert \,.$$

	Now observe that by removing vertex $v$ from $C_i$, we break $C_i$ into at least one and at most $\vert N_{C_i}(v)\vert$ connected components.
	Last inequality tells us that the average size of these components is strictly larger than $|C_j|$.
	By the pigeonhole principle, one of these components, say $T_v$ verifies $\vert V(T_v)\vert\geq \vert C_j \vert $.
	In particular, we have $|C_i| > |C_j|$.

	We are now going to transform $P$ into another connected $k$-partition $\Pi'$ of $T$ such that $\Pi<\Pi'$.

	Let $A = (C_i \setminus V(T_v)) \cup C_j$ and $B = V(T_v)$.
	We define $\Pi'$ in the following way: $\Pi' = (\Pi \setminus \{ C_i, C_j\}) \cup \{ A, B\}$.
	See \Cref{fig:treestransformation} for an illustration with~$k=3$.

	Notice that $T[A]$ induces a connected subgraph of $T$ since $v$ is adjacent to some vertex in~$C_j$.
	Then, $\Pi'$ is a connected $k$-partition of $T$, since $\Pi$ is a connected $k$-partition of $T$ and both $T[A]$ and $T[B]$ induce connected subgraphs of~$T$.
	Furthermore, we observe the following:
	\begin{itemize}
		\item since $V(T_v) \subset C_i$, we have $|A| = |C_i| - |V(T_v)| + |C_j| > |C_j|$;
		\item it follows from the argument above that $|B| = |V(T_v)| > |C_j|$.
	\end{itemize}
	Then, let $S$ and $S'$ be the size tuples of $\Pi$ and $\Pi'$, respectively.
	Since $|C_j|<|A|,|B|$ and $|C_j| < |C_i|$, we conclude that $S < S'$.
	Hence, we have that $\Pi < \Pi'$.
	Furthermore, $\Pi'$ can be obtained in time $\mathcal{O}(|V|)$.
\end{proof}

\Cref{not sat implies not lex} implies the following corollary.

\begin{corollary}
	\label{unif partition is k comm}
	Let $T=(V,E)$ be a tree such that $|V|\ge k$, for some integer $k\ge 2$.
	Let $\Pi=\{C_1,\ldots,C_k\}$ be a connected uniform $k$-partition of~$T$.
	Then, $\Pi$ is a connected generalized $k$-community structure of~$T$.
\end{corollary}

\begin{proof}
	Assume by contradiction that $\Pi$ does not form a generalized $k$-community structure of~$T$.
	Thus, there exists some $v\in V$ that does not satisfy \eqref{prop:k-comm1} with respect to~$\Pi$.
	Then, \Cref{not sat implies not lex} implies that there exists a $k$-partition $\Pi'$ of $T$ such that $\Pi<\Pi'$.
	A contradiction to the assumption that $\Pi$ is a connected uniform $k$-partition of~$T$.
\end{proof}\newpage

Let us call the transformation from a $k$-partition $\Pi$ to a $k$-partition $\Pi'$ of $T$, such that $\Pi < \Pi'$, as described at the end of the proof of \Cref{not sat implies not lex} an \textsc{Improvement Procedure}.
Then, we define an \textsc{Improvement Algorithm} as follows.
Assume, that we are given a positive integer $k$, a tree $T=(V,E)$ such that $|V|\ge k$, and a connected $k$-partition $\Pi$ of~$T$.
While there exists $v\in V$ that is not satisfied with respect to $\Pi$, we apply the \textsc{Improvement Procedure}.
In this way, we obtain a connected generalized $k$-community structure of~$T$.
The algorithm terminates, since at each iteration of the while loop, the connected $k$-partition of $T$ gets updated and its size-tuple increases lexicographically.
Next, we show that the algorithm runs in time $\mathcal{O} (k^2\cdot|V|^2)$ by presenting apolynomial  upper bound on the number of iterations of the \textsc{Improvement Algorithm}.

To this end, let us introduce several notions.
For two positive integers $k,n\ge 2$, a \emph{$(k,n)$-tuple} $t$ is an ordered tuple of $k$ non-negative integers such that $\sum_{i=1}^k t[i] = n$.
Notice that, given a connected $k$-partition $\Pi = \{C_1,\ldots,C_k\}$ of a tree $T =(V,E)$ on $n$ vertices, the size tuple of $\Pi$ is a $(k,n)$-tuple.
Let $t$ be a $(k,n)$-tuple, $x, y \in \intInterval{k}$ such that $x < y$ and $ t[x] + 2 \le t[y] $, and $\delta\in \intInterval{t[y] - t[x] -1}$.
An \emph{$(x,y)$-change of a $(k,n)$-tuple $t$ of order $\delta$} is a $(k,n)$-tuple whose elements correspond exactly to the elements of the multiset $\{t[i] : i \in \intInterval{k} \setminus \{x,y\} \} \cup \{t[x] + \delta, t[y] - \delta\}$.
We say that a $(k,n)$-tuple $t'$ is a \emph{$2$-change of $t$ of order $\delta$} if $t'$ is an $(x,y)$-change of $t$ of order $\delta$, for some $x<y \in \intInterval{k}$ (see \Cref{fig:tuples} for an illustration).
If $\delta$ is clear from the context, we may simply say a $2$-change of~$t$.
Since $t$ is a $(k,n)$-tuple, then $2$-change of $t$ is a $(k,n)$-tuple as well.

Observe that if $t$ and $t'$ are $(k,n)$-tuples such that $t<t'$ and $t$ and $t'$ share all but two of their elements (disregarding their positions in the tuples), then $t'$ is a $2$-change of~$t$.
For completeness, we state this formally.

\begin{lemma}
	\label{2change}
	Let $t $ and $t' $ be two $(k,n)$-tuples that differ in exactly two elements and such that $t<t'$.
	Then $t'$ is a $2$-change of~$t$.
\end{lemma}
\begin{proof}
	By definition, $t<t'$ implies that there exists $j\in \intInterval{k}$ such that $t[j] < t'[j]$ and $t[i] = t'[i]$, for all $1\le i < j$.
	Then $t'[j] = t[j]+ \delta$, for some $\delta \ge 1$.
	Furthermore, by assumption $t$ and $t'$ differ in exactly two elements and since $\displaystyle \sum_{i=1}^k t[i] = \sum_{i=1}^k t'[i] = n$, thus there must exist some $j' \in \intInterval{k}$ such that $t'[j'] = t[j'] - \delta$.
	Moreover, since $t[i] = t'[i]$, for all $1\le i < j$, we have $j < j'$ and $t[j'] = t'[j'] + \delta \ge t'[j] + \delta = t[j]+ \delta + \delta = t[j]+ 2\delta \ge t[j]+ 2 $.
	Hence, the elements of $t'$ exactly correspond to the elements of the multiset $\{t[i] : i \in \intInterval{k} \setminus \{j,j'\} \} \cup \{t[j] + \delta, t[j'] - \delta\}$, where $j, j' \in \intInterval{k}$ are such that $j < j'$ and $t[j]+ 2 \le t[j']$.
	Thus, $t'$ is a $2$-change of $t$ of order $\delta$.
\end{proof}

The \textsc{Improvement Procedure} modifies a connected $k$-partition $\Pi = \{C_1,\ldots,C_k\}$ of $T$ in the following way.
If there exists a vertex $v \in V$ in some community $C_i$, for some $i\in\intInterval{k}$, which does not satisfy \cref{prop:k-comm2} (resp.\ \cref{prop:k-comm1}) for some community $C_j \in \Pi$ (which is possible only if $ \vert (\vert C_i \vert - \vert C_j \vert )\vert \ge 2$), $i\neq j\in \intInterval{k}$, we transform $\Pi$ into another connected $k$-partition $\Pi'$ of $T$ such that $\Pi<\Pi'$, by  moving some vertices from community $C_i$ to community~$C_j$. Notice that this transformation modifies the sizes of exactly two communities, namely $C_i$ and~$C_j$. Let $S$ and $S'$ be the size tuples of the $k$-partitions $\Pi$ and $\Pi'$, respectively. By \Cref{2change}, we know that $S'$ is a $2$-change of~$S$. Since we do this transformation in \Cref{not sat implies not lex} as long as there exists a vertex which is not satisfied, this actually corresponds to a sequence of $2$-changes.

Next, we present an upper bound on a sequence of $2$-changes which implies in an upper bound on the number of iterations of our \textsc{Improvement Algorithm}.
Let $t$ be a $(k,n)$-tuple and let $t'$ be a $2$-change of $t$ of order~$\delta$.
We may assume without loss of generality, that $\delta \leq (t[y] - t[x])/2$.
Indeed, assume that $\delta > (t[y] - t[x])/2$.
Let $\delta' = t[y] - t[x] - \delta$.
Since $(t[y]-t[x])/2 < \delta \leq t[y] - t[x] - 1 $, we have $\delta ' = t[y] - t[x] - \delta < (t[y]-t[x])/2$ and $1 \le \delta' $.
We also have $t[x] + \delta' = t[y] - \delta$, $t[y] - \delta' = t[x] + \delta$.
Hence, if $\delta > (t[y] - t[x])/2$ and $t'$ is a $2$-change of $t$ of order $\delta$, we can define $\delta' < (t[y]-t[x])/2$ as described above such that $t'$ is a $2$-change of $t$ of order $\delta'$.

Let $w_j(t) = \sum_{i=1}^{j} t[i]$, i.e.\ the sum of the first $j$ elements of a  $(k,n)$-tuple~$t$. We show that if $t'$ is a $2$-change of $t$ of order $\delta$, then $w_j(t) \leq w_j(t')$, for all $j \in \intInterval{k}$.

\begin{lemma}\label{lemma: w_j (t') >= w_j(t)}
	Let $t$ be a $(k,n)$- tuple and let $t'$ be a $2$-change of $t$ of order~$\delta$.
	Then $w_j(t) \leq w_j(t')$, for all $j \in \intInterval{k}$.
\end{lemma}
\begin{proof}
	Since $t'$ is a $2$-change of $t$ of order $\delta$, we know by definition that there exist $x<y \in \intInterval{k}$ such that $t'=\{t[i] : i \in \intInterval{k} \setminus \{x,y\} \} \cup \{t[x] + \delta, t[y] - \delta\}$.
	Let $x'$ be the smallest index such that $t'[x'] =t[x] + \delta$ and $y'$ be the largest index such that $t'[y'] = t[y] - \delta$.
	Recall that we may assume without loss of generality that $\delta \leq (t[y] - t[x])/2$, or equivalently $t[x] + \delta \leq t[y] - \delta$, and thus, since $t'$ is ordered, we have $x' < y'$.

	We want to show that $w_j(t) \leq w_j(t')$, for all $j \in \intInterval{k}$.
	We distinguish five cases depending on the index~$j$ (see \Cref{fig:tuples}).
	\begin{figure}[!ht]
		\centering
		\includegraphics[width=0.80\linewidth]{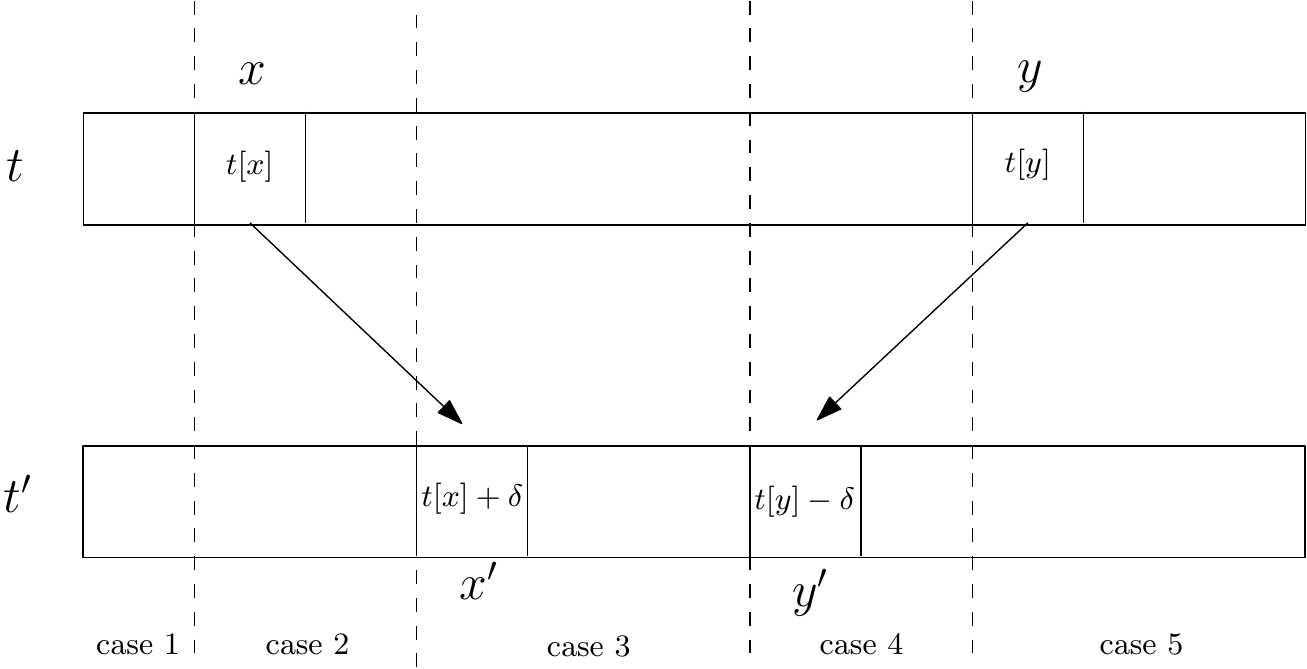}
		\caption{A $(k,n)$-tuple $t$ and an $(x,y)$-change $t'$ of $t$ of order $\delta$ and a graphical representation of the five cases considered in the proof of \Cref{lemma: w_j (t') >= w_j(t)}.}\label{fig:tuples}
	\end{figure}
	\setlist[description]{font=\normalfont}
	\begin{description}
		\item[Case 1.] \label{case 1} $1 \le j \le x - 1$.
		      Since $ t[j] = t'[j]$, we have $w_j(t) = w_j(t')$.
		\item[Case 2.] \label{case 2} $x \leq j \le x'-1$.
		      In this case, we have
		      $$w_j(t') = w_j(t) - t[x] + t[j+1] \ge w_j(t)\,, $$
		      since $ t[j+1] \ge t[x] $.
		\item[Case 3.] \label{case 3} $x' \le j \le y'-1$.
		      We have
		      $$w_{j}(t') =w_j(t) - t[x] + t'[x'] =  w_j(t) - t[x] + t[x] + \delta = w_{j}(t) + \delta\,.$$
		      Since $\delta \ge 1$ we obtain, $w_j(t) \leq w_j(t')$.
		\item[Case 4.] \label{case 4} $y' \le j \le y-1$.
		      In this case, we have
		      \begin{align*}
			       & w_j(t') = w_{y'-1} (t')+ t'[y'] +\sum_{i=y'+1}^{j} t'[i]                          \\
			       & =w_{y'-1} (t) + \delta + t'[y'] + \sum_{i=y'+1}^{j} t'[i] =
			      w_{y'-1} (t) + \delta + t[y] - \delta + \sum_{i=y'}^{j-1} t[i]                       \\
			       & = w_{y'-1} (t) + t[y] + \sum_{i=y'}^{j-1} t[i] \ge w_{j-1}(t) + t[j] = w_j(t) \,,
		      \end{align*}
		      where $w_{y'-1} (t') = w_{y'-1} (t) + \delta$ follows from the previous case and $ t[j] \le t[y]$ since $ j<y$ and $t'$ is ordered.
		\item[Case 5.] \label{case 5} $y \le j \le k$.
		      In this case, we have
		      \begin{align*}
			      w_j(t') = w_{j} (t) - t[x] - t[y] + (t[x] + \delta) +(t[y] - \delta) = w_j(t)\,.
		      \end{align*}
	\end{description}
	Hence, for all $j\in\intInterval{k}$, we have $w_j(t) \leq w_j(t')$, as claimed.
\end{proof}

In the next lemma, we derive an upper bound on $w_j(t)$, for every $j \in \intInterval{k}$.

\begin{lemma}\label{w_j is not too large}
	Let $k$ and $n$ be two positive integers.
	Let $t$ be a $(k,n)$-tuple.
	Then $w_j(t) \leq j \cdot \frac{n}{k}$, for all $j \in\intInterval{k}$.
\end{lemma}
\begin{proof}
	Assume by contradiction that there exists $j \in \intInterval{k}$ such that $w_j(t) > j \cdot \frac{n}{k}$.
	If $j = k$, then $t$ would not be a $(k,n)$-tuple.
	So we may assume that $j < k$.
	Since $w_j(t) = \sum_{i=1}^j t[i] > j \cdot \frac{n}{k}$, there must exist $x \in \intInterval{j}$ such that $t[x] > \frac{n}{k}$.
	We also have that $\sum_{i=j+1}^k t[i] = n - w_j(t) < n - j \cdot \frac{n}{k} = (k-j) \cdot \frac{n}{k}$, and thus, there exists $y \ \in \intInterval{k} \setminus \intInterval{j}$ such that $t[y] < \frac{n}{k}$.
	However, since $x < y$ and $t[x] > t[y]$, this contradicts the fact that $t$ is an ordered tuple.
\end{proof}

\begin{lemma}
	\label{kn-tuples_bounded}
	Let $k$ and $n$ be two positive integers.
	Let $\mathcal{S} = \langle t_1, \dots, t_\ell \rangle$ be a sequence of ordered $(k,n)$-tuples such that, for every two consecutive tuples $t_{i-1}$ and $t_i$ in $\mathcal{S}$, it holds that $t_{i}$ is a $2$-change of $t_{i-1}$.
	Then $\ell \leq 1 + \frac{k+1}{2} \cdot n$.
\end{lemma}
\begin{proof}
	For all $j \in \intInterval{k}$, we define a set $S_j$ of $(k,n)$-tuples $t_i \in \mathcal{S}$, such that $i > 1$ and $t_{i}$ is an $(j,j')$-change of $t_{i-1}$, for $j<j' \in \intInterval{k}$, and $j$ is maximum with this property.
	By the choice of $j$, we have that every tuple in $\mathcal{S} \setminus \{t_1\}$ belongs to exactly one set $S_j$, for some $j \in \intInterval{k}$.
	We prove by induction on the length $\ell$ of $\mathcal{S}$ that $|S_j| \leq w_j(t_\ell)$, for every $j \in \intInterval{k}$.
	If $\mathcal{S}$ contains at most one tuple, then $|S_j| = 0$, for all $j \in  \intInterval{k}$, and the claim holds.
	Suppose now that the claim remains true whenever the length of $\mathcal{S}$ is less than $\ell$, for some $\ell \in \mathbb{N}$.

	We define $\mathcal{S}' = \mathcal{S} \setminus \{t_\ell\}$ and, similarly as above, the set $S'_j$ for all $j \in \intInterval{k}$, is the set of all $(k,n)$-tuples $ t_i \in \mathcal{S'}$, such that $i > 1$ and $t_{i}$ is an $(j,j')$-change of $t_{i-1}$, for $j<j' \in \intInterval{k}$,  and $j$ is maximum with this property.
	Let $x<y \in \intInterval{k}$ and $\delta \in \intInterval{t[y] - t[x]-1}$ such that $t_\ell$ is an $(x,y)$-change of $t_{\ell-1}$ of order~$\delta$.
	Without loss of generality, we may assume $x$ to be maximum with this property.
	By the choice of $x$, we have $t_\ell \in S_x$.
	Recall that we may assume, without loss of generality, that $\delta \leq (t_{\ell-1}[y] - t_{\ell-1}[x])/2$.
	Also notice that for all $j \in \intInterval{k} \setminus \{x\}$, it holds that $S_{j} = S'_{j} $, and hence the induction hypothesis and \Cref{lemma: w_j (t') >= w_j(t)} imply that $|S_{j}| = |S'_j| \leq w_j(t_{\ell-1}) \leq w_j(t_\ell)$.
	By the choice of $x$, we have that $t_{\ell-1}[x] < t_{\ell-1}[x+1]$.
	In particular, we can deduce that $t_\ell[x] = \min \{t_{\ell-1}[x+1], t_{\ell-1}[x] + \delta\} > t_{\ell-1} [x]$.
	We also note that $t_\ell[i] = t_{\ell-1}[i]$, for all $x>i \in \intInterval{k}$.
	Consequently, $|S_x| = |S'_x| + 1 \leq w_x(t_{\ell-1}) +1= w_{x-1}(t_{\ell-1}) + t_{\ell-1}[x] +1 \le w_{x-1}(t_{\ell-1}) + t_{\ell}[x] = w_x(t_\ell)$, which proves the induction hypothesis.

	We can now derive an upper bound on the length $\ell$ of $\mathcal{S}$ as follows:
	\begin{align*}
		 & \ell = 1 + \sum_{j=1}^k |S_j|                                                       \\
		 & \leq 1 + \sum_{j=1}^k w_j(t_\ell)  \text{ (see above)}                              \\
		 & \leq 1 + \sum_{j=1}^k j \cdot \frac{n}{k}  \text{ (by \Cref{w_j is not too large})} \\
		 & =  1 + \frac{k+1}{2} \cdot n\,.
	\end{align*}
\end{proof}

This leads to the following result.

\begin{theorem}\label{thrm: gen k-comm}
	Let $k \geq 2$ be an integer and let $T=(V,E)$ be a tree such that $|V|\ge k$.
	Then, there exists a connected generalized $k$-community structure of $T$ that can be computed in $\mathcal{O}(k^2\cdot|V|^2)$ time.
\end{theorem}

\begin{proof}
	Notice that a connected $k$-partition $\Pi$ of $T$ can simply be obtained by deleting any $k-1$ edges.
	We then apply the \textsc{Improvement Algorithm} starting with~$\Pi$.
	If all the vertices of $T$ are satisfied with respect to $\Pi$, we are done.
	Otherwise, we apply the \textit{``while''} loop from the \textsc{Improvement Algorithm} to compute a connected $k$-partition $\Pi'$ of $T$ such that $\Pi<\Pi'$.
	The sequence of size tuples corresponding to the successive partitions obtained that way is a lexicographic sequence of $(k,n)$-tuples.
	We know from \Cref{2change} that for every two consecutive $(k,n)$-tuples $t$ and $t'$ in the sequence, where $t<t'$, it holds that $t'$ is a $2$-change of~$t$.
	By \Cref{kn-tuples_bounded}, the size of such a sequence is at most $1 + \frac{k+1}{2} \cdot n=1 + \frac{k+1}{2} \cdot |V| = \mathcal{O} (k |V|)$, and so is therefore the number of iterations of the \textsc{Improvement Algorithm}.
	Then, by \Cref{check if kcs linear time,not sat implies not lex} each iteration can be carried out in time $\mathcal{O}(k |V|)$ in trees.
	We conclude that the complexity of the \textsc{Improvement Algorithm} is $\mathcal{O}(k^2 \cdot |V|^2)$.
	Hence, the statement follows.
\end{proof}

\begin{figure}[ht]
	\begin{center}
		\includegraphics[width=0.9\linewidth]{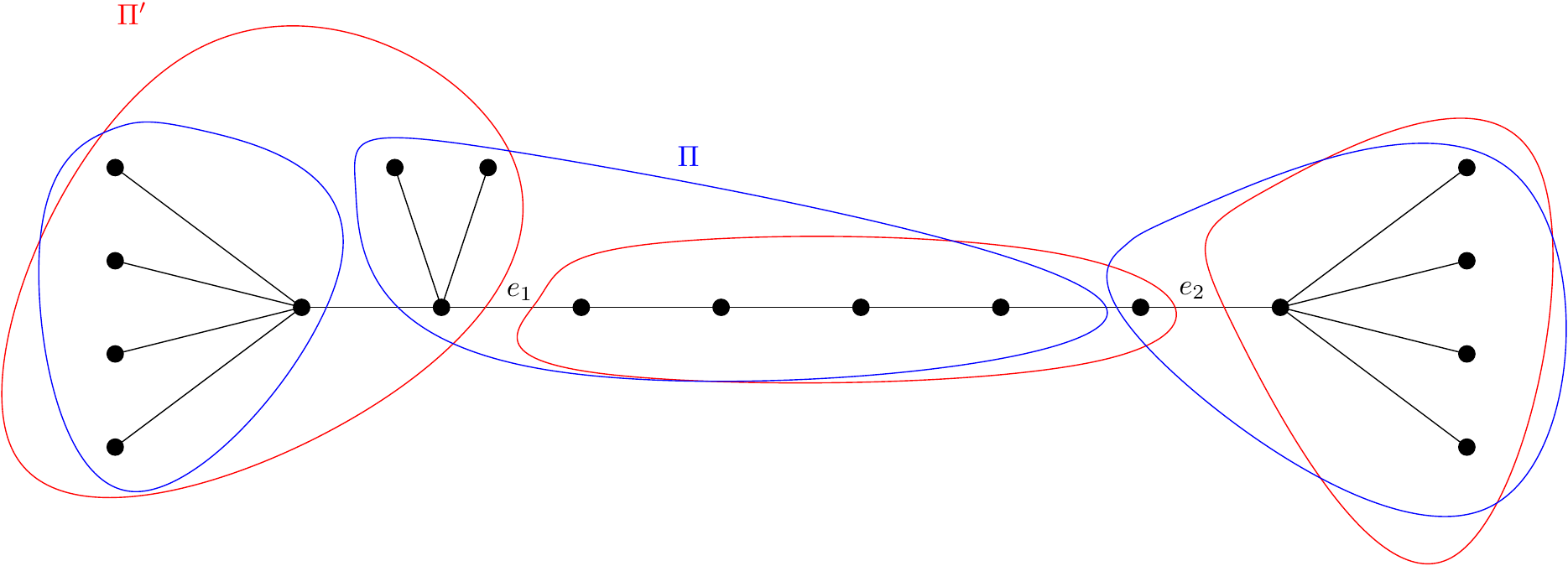}
	\end{center}
	\caption{A tree $T$ and a generalized $3$-community structure $\Pi'$ returned by the \textsc{Improvement Algorithm} represented by red bags and a $3$-community structure $\Pi$ represented by blue bags which is lexicographically larger than $\Pi'$.}
	\label{fig:3-comm which is not uniform}
\end{figure}

Notice that the \textsc{Improvement Algorithm} does not necessarily return a uniform $k$-partition.
To illustrate this, in \Cref{fig:3-comm which is not uniform} we present a tree $T$ and a connected $3$-community structure $\Pi'$, represented by red bags, that can be obtained by the \textsc{Improvement Algorithm} (it suffices to start with $\Pi'$, which the algorithm does not change).
A different connected $3$-community structure $\Pi$, represented by blue bags, is lexicographically larger than $\Pi'$.
This implies that $\Pi'$ is not uniform.
Hence, it would be interesting to consider the following open question.
\begin{quote}
		What is the time complexity of computing a connected uniform $k$-partition of a tree?
\end{quote}

We now give a necessary and sufficient condition for a tree to admit a $k$-community structure for any integer $k\geq 2$, and show how to obtain it, if it exists, in linear time.

\begin{theorem}\label{k-cs not small in trees poly}
	Let $k\geq 2$ be an integer.
	A tree $T$ admits a $k$-community structure if and only if $T$ contains a matching of size~$k$.
	Furthermore, if such a $k$-community structure exists, it can be found in time $\mathcal{O}(k^2\cdot |V|^2)$.
\end{theorem}

\begin{proof}
	By \Cref{k-matching}, we know that if $T$ admits a $k$-community structure, then $T$ contains a matching of size~$k$.
	Conversely, assume that $T$ contains a matching $M$ of size $k$, for $k\ge 2$.
	We construct a connected $k$-partition $\Pi=\{C_1,\ldots,C_k\}$ of $T$ such that $|C_i|\geq 2$ for $i \in \intInterval{k}$ as follows.
	The endpoints of the $k$ edges in $M$ define the $k$ initial subsets of the partition, where two vertices belong to the same subset if and only if they are endpoints of a same edge in~$M$.
	We then assign iteratively each remaining vertex to a subset, chosen in such a way that each subset always induces a connected subgraph.
	This results in a connected $k$-partition $\Pi = \{C_1,\ldots,C_k\}$ of $T$ such that $|C_i|\geq 2$ for $i \in \intInterval{k}$.
	Hence, similarly as in the proof of \Cref{thrm: gen k-comm}, we can use the \textsc{Improvement Algorithm} to obtain a connected $k$-community structure of $T$ in time $\mathcal{O}(k^2\cdot \vert V \vert ^{2})$.
\end{proof}

Now, let us consider forests.
\Cref{thrm: gen k-comm} can be extended to forests in a rather straightforward way.
Indeed, let $F=(V,E)$ be a forest with $\vert V\vert\ge k$ and containing $m$ connected components $T_1,\ldots,T_m$.
If $m\geq k$, then  $C_1=V(T_1),\ldots,C_{k-1}=V(T_{k-1}),C_k=V(T_k)\cup\ldots\cup V(T_m)$ forms a generalized $k$-community structure.
So we may assume that~$m<k$.
Let $T_1,\ldots,T_{m'}$, $m'<m$, be the connected components of $F$ containing each exactly one vertex.
We set $C_1=V(T_1),\ldots,C_{m'}=V(T_{m'})$, and in order to find a generalized $k$-community structure in $F$, it is enough to find integers $s_{1},s_{2},\ldots,s_{m-m'}$ such that $s_{1}+s_{2}+\ldots+s_{m-m'}=k-m'$ and $s_{1}\le |V(T_{m'+1})|, \ldots, s_{m-m'}\le |V(T_m)|$.
Such integers always exist since $k-m' \le \vert V\vert-m' $, and, they can be determined in linear time.
Then, we only need to compute a generalized $s_j$-community structure in $T_{m'+j}$, for each $j\in\intInterval{m-m'}$.
This can be done using \Cref{thrm: gen k-comm}.
Thus, we obtain the following.

\begin{corollary}\label{forest generalized com-struct}
	Let $F=(V,E)$ be a forest such that $\vert V\vert\ge k$, for some integer $k\ge 2$.
	Then, $F$ admits a  generalized $k$-community structure that can be computed in time $\mathcal{O}(k^2\cdot |V|^2)$.
\end{corollary}

\Cref{k-cs not small in trees poly} can also be extended to forests as follows.

\begin{theorem}
	Let $k\ge 2$ be a positive integer.
	Let $F=(V,E)$ be a forest  and let $I \subseteq V$ be the set of isolated vertices of~$F$.
	Then, $F$ admits a $k$-community structure if and only if $F$ contains a matching $M$ such that $|M| + \lfloor |I|/2 \rfloor \geq k$.
	Furthermore, if such a $k$-community structure exists, it can be found in time $\mathcal{O}(|V|^{2})$.
\end{theorem}

\begin{proof}
	We first show necessity.
	Let $\Pi=\{C_1, \dots,C_k\}$ be a $k$-community structure in $F$ and let $M\subseteq E$ be empty.
	For every community $C_i$, $i\in \{1,\ldots,k\}$, such that $F[C_i]$ contains at least one edge, we assign one edge of $F[C_i]$ to~$M$.
	If $r$ communities induce graphs with at least one edge, then $M$ is a matching of size~$r$.
	Let $\ell = k - r$ be the number of communities that induce edgeless graphs.
	Recall that every community contains at least $2$ vertices.
	Observe that, if $C_j$, for any $j\in \{1,\ldots,k\}$, is a community inducing an edgeless graph and if $v \in C_j$, then $v$ is isolated in $F$, otherwise $v$ cannot be satisfied with respect to~$\Pi$.
	We conclude that $F$ contains at least $2\ell$ isolated vertices.
	Hence, $F$ contains a matching $M$ such that $|M| + \lfloor |I|/2 \rfloor \geq k$.

	We now show sufficiency.
	If $k\leq \lfloor |I|/2 \rfloor$, it is enough to partition $I$ into $k$ sets $C_1,\ldots,C_k$, each containing at least two vertices, and to add the remaining vertices of $F$ to one of these sets.
	Since there are no edges between any two distinct communities, this gives us a $k$-community structure.\newline
	So we may assume now that $k> \lfloor |I|/2 \rfloor$.
	We start by creating $k$ empty sets $C_1, \dots, C_k$.
	All along our procedure, we make sure that for every vertex $v \in C_i$, it holds that

	\begin{equation}
		\label{invariant}
		\frac{\vert N_{C_i} (v)\vert}{\vert C_i \vert -1} \ge \frac{\vert N_{C_j} (v)\vert}{\vert C_j \vert}\ \text{for all $i,j \in \intInterval{k}$, $i\neq j$}\,.
	\end{equation}

	Let $\ell = \lfloor |I|/2 \rfloor$.
	If $\ell\geq 1$, we assign the vertices of $I$ to the sets $C_1,\dots,C_\ell$, such that every $C_i$ with $i \in \intInterval{\ell}$ contains at least $2$ vertices of~$I$.
	Notice that \cref{invariant} still holds, since after assigning the vertices of $I$ to the sets $C_1,\dots,C_\ell$, each $C_i$, for $i \in \intInterval{\ell}$, consists of a subset of connected components of~$T$.
	If $\ell=0$, all sets $C_i$, $i\in \intInterval{k}$, remain empty at this stage.
	Recall that by assumption, $F$ contains a matching $M$ of size at least $k - \ell$.
	Let $M$ be a matching in $F$ of size exactly $k - \ell$.

	Then, if there exists no connected component $T=(V(T),E(T))$ of $F\backslash I$ such that $T$ verifies $|E(T) \cap M| = 1$, and there exists a connected component $T'=(V(T'),E(T'))$ of $F \backslash I$ such that $T'$ that verifies $|E(T') \cap M| = 0$, we remove one edge of the matching $M$ from some connected component $T''$ of $F\backslash I$ that verifies $|E(T'') \cap M| \geq 2$ (such a $T''$ exists since $|M|=k-\ell>0$), and add some edge of $T'$ to $M$ instead.
	Note that $M$ remains a matching.
	After this procedure, also notice that at this point, either there is a connected component $T = (V (T),E(T))$ of $F$ such that $T$ verifies $|E(T) \cap M| = 1$, or all connected components $T$ of $F\setminus I$ verify $|E(T) \cap M| \geq 2$.

	For every connected component $T=(V(T),E(T))$ of $F$ such that $|E(T) \cap M| = 1$, if such a component exists, we add every vertex of $T$ to a set $C_j$ such that $j$ is the smallest possible index from the set  $\{\ell+1,\ldots,k\}$ and $C_j$ is still empty.
	Notice that such a set $C_j$ always exists, since $M$ has size $k-\ell$.
	Furthermore, notice that after assigning every vertex of $T$ to $C_j$, the set $C_j$ contains at least $2$ vertices and (\ref{invariant}) is satisfied.
	Then, for connected components with $|E(T) \cap M| = p \geq 2$, we use \Cref{k-cs not small in trees poly} to compute a connected $p$-community structure $\{P_1, \dots, P_p\}$ of $T$ in time $\mathcal{O}(k^2\cdot|V(T)|^{2})$, and assign the vertices of each set $P_i$, $i=1,\ldots,p$, to a different set $C_j$ that is still empty, for $j \in \{\ell+1,\ldots,k\}$.
	Again, such sets $C_j$ always exist, since $M$ has size $k-\ell$.
	Furthermore, (\ref{invariant})  is satisfied, since $\{P_1, \dots, P_p\}$ is a $p$-community structure of~$T$.
	Then, up to this point we have that every set $C_i$, $i\in \intInterval{k}$, contains at least two vertices and that (\ref{invariant}) is satisfied.

	In order to assign the remaining vertices (either isolated vertices or connected components of $T$ containing no edges of $M$) to the sets $C_1,\ldots,C_k$, we distinguish two cases.
	If $\ell\geq 1$ or if there exists at least one connected component $T=(V(T),E(T))$ of $F$ that verifies $|E(T) \cap M| = 1$, we simply add all of the unassigned vertices to $C_1$ and inequality (\ref{invariant}) remains satisfied, since in this case $C_1$ consists of a subset of connected components of~$F$.
	If $\ell = 0$ (i.e.\ $\vert I\vert =0$ or $\vert I\vert =1$) and no connected component $T=(V(T),E(T))$ of $F$ verifies $|E(T) \cap M| = 1$, there exists at most one unassigned vertex, namely an isolated vertex, say~$v$. If so, we simply add it to a community among $C_1, \dots, C_k$ containing the smallest number of vertices, say~$C_i$.
	Let us denote $C_i' = C_i \cup \{v\}$.
	Then, for all vertices $u \in C_i$ and all $j\in \intInterval{k}\backslash i$, we have $\vert N_{C_i'}(u)\vert \geq \vert N_{C_j}(u)\vert$ (recall that we used \Cref{k-cs not small in trees poly} to compute a connected community structure and hence every vertex in $T$ has at most one neighbor outside of its own community), and $\vert C_i'\vert-1=\vert C_i \vert\leq \vert C_j \vert$.
	Therefore, inequality $\frac{\vert N_{C_i'} (u)\vert}{\vert C_i' \vert -1} \ge \frac{\vert N_{C_j} (u)\vert}{\vert C_j \vert}$ still holds.

	Let us now analyse the complexity of our procedure described above.
	Computing $I$ and $M$ takes linear time, and so do the first steps of the algorithm that assign the vertices of $I$ and the connected components of $F$ intersecting $M$ on a single edge to some sets $C_i$, $i\in\{1,\ldots,k\}$.
	Let $T_1, \ldots, T_t$ denote the connected components of $F$ containing at least $2$ edges of $M$, and let $p_i = |E(T_i) \cap M|$ for $i=1,\ldots,t$.
	Since for each connected component $T_i$, for $i \in \intInterval{t}$ we can compute a $p_i$-community structure in time $\mathcal{O}(k^2\cdot |V(T)|^{2})$, the claimed complexity follows.
\end{proof}

\section{Threshold graphs}
\label{sec-thresh}

In this section, we focus on the existence of a (generalized) $2$-community structure in threshold graphs.
A \textit{threshold graph} is a graph that can be constructed from the one-vertex graph by repeatedly adding an isolated vertex or a universal vertex.
An equivalent definition (see~\cite{threshold}) is that it is a graph $G=(V,E)$, where $V$ can be partitioned into a clique $Q=\{v_1,\ldots,v_q\}$ and a stable set $S=\{w_1,\ldots,w_s\}$ such that $N(w_i)\subseteq N(w_{i+1})$ for $i\in \intInterval{s-1}$ (and therefore also, without loss of generality, $N(v_i)\supseteq N(v_{i+1})$ for $i\in \intInterval{q-1}$).
Notice that a threshold graph $G=(V,E)$, which is not connected, corresponds to the union of a connected threshold graph $H=(V_H,E)$ and a set $I$ of isolated vertices.

Let us start by showing that every threshold graph $G$ admits a generalized $2$-community structure.
Observe that in any threshold graph, $N(w_1)$ corresponds to the set of all universal vertices of~$G$.
Notice that if $G$ does not contain any universal vertex, then $G$ is disconnected and $w_1$ is an isolated vertex.
Consequently, as pointed out in \Cref{2commUniv}, every threshold graph containing at least two vertices admits a generalized $2$-community structure $\Pi=\{\{w_1\},V\backslash \{w_1\}\}$.

For $2$-community structures, we obtain the following for the case of connected threshold graphs.

\begin{theorem}
	\label{thm:threshold}
	Let $G=(V,E)$ be a connected threshold graph with $n\geq 4$ vertices.
	Then, $G$ admits a $2$-community structure if and only if $G$ is not isomorphic to the star $S_{n-1}$.
	Furthermore, if it exists, a $2$-community structure can be found in time $\mathcal{O}(|E|)$.
\end{theorem}

\begin{proof}
	Necessity follows form the fact that the star $S_{n-1}$ does not admit any $2$-community structure.

	We now show sufficiency.
	Since by assumption $G$ is connected and not isomorphic to the star $S_{n-1}$, it follows that~$q>1$.
	We claim that $\Pi=\{C_1,C_2\}$ with  $C_1=\{v_1,w_1,\ldots,w_k\}, C_2=\{v_2,\ldots, v_q, w_{k+1},\ldots, w_s\}$ and $k=\max\{i:d(w_i)\leq \frac{n-1}{i}\}$ is a $2$-community structure in~$G$.
	On \Cref{thresholdgraph2commEx}, we show an example of such a $2$-community structure $\Pi=\{C_1,C_2\}$ on a threshold graph $G$ on $10$ vertices with $q = s= 5$ and $k=\max\{i:d(w_i)\leq \frac{9}{i}\}=3$.

	\begin{figure}[!ht]
		\centering
		\includegraphics[width=0.5\linewidth]{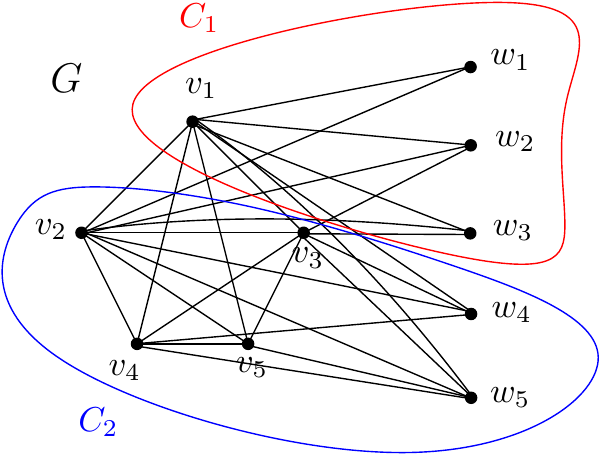}
		\caption{An example of a threshold graph $G$ and a $2$-community structure $\Pi=\{C_1,C_2\}$ in $G$ constructed as described at the beginning of the proof of \Cref{thm:threshold}.}
		\label{thresholdgraph2commEx}
	\end{figure}

	In order to show that the partition $\Pi=\{C_1,C_2\}$ is a $2$-community structure, we need to show that $|C_1|\geq 2$, $|C_2|\geq 2$, and that all vertices are satisfied with respect to~$\Pi$.
	First, notice that we may assume without loss of generality that $d(w_s)=q$.
	Indeed, if $d(w_s)<q$, then $v_q$ has no neighbour in $S$, and it may then be considered as a vertex of $S$, adjacent to all the vertices in~$Q$.
	This also allows us to assume that $s\geq 1$ and hence, that $w_1$ actually exists.
	Second, we trivially know that $1\in \{i:d(w_i)\leq \frac{n-1}{i}\}$.
	Therefore, $|C_1|\geq 2$.
	Assume now that $s\in \{i:d(w_i)\leq \frac{n-1}{i}\}$.
	Then, $d(w_s)=q\leq \frac{n-1}{s}$, which means that $q\cdot s\leq n-1=q+s-1$.  It is easy to see that the only possibility for this to happen, knowing that $q\geq 2$ and $s\geq 1$, is that~$s=1$. Since $s=1$ and $d(w_s)=q$, we conclude that $G$ is a clique. In this case $q \geq 3$ and we have $C_1=\{v_1,w_1\}$, $C_2=\{v_2,\ldots, v_q\}$ and $\Pi$ is a $2$-community structure (notice that $|C_2|\geq 2$ since $q=n-s\geq 3$).

	Assume now that $s\not\in \{i:d(w_i)\leq \frac{n-1}{i}\}$, which directly implies that $|C_2|\geq 2$.
	We need to prove that all vertices are satisfied with respect to~$\Pi$.
	We start with the vertices in~$S$.
	Recall that they all have exactly one neighbour in $C_1$, namely~$v_1$.

	\begin{itemize}
		\item[$\bullet$] $w_k$: $\frac{|N_{C_2}(w_k)|}{|C_2|}=\frac{d(w_k)-1}{n-k-1}\leq \frac{\frac{n-1}{k}-1}{n-k-1}=\frac{1}{k}=\frac{|N_{C_1}(w_k)|}{|C_1|-1}$;
		\item[$\bullet$] $w_i$ for $i<k$: it immediately follows from the previous case, since $|N_{C_2}(w_i)|\leq |N_{C_2}(w_k)|$ and $|N_{C_1}(w_i)|= |N_{C_1}(w_k)|$;
		\item[$\bullet$] $w_{k+1}$: $\frac{|N_{C_2}(w_{k+1})|}{|C_2|-1}=\frac{d(w_{k+1})-1}{n-k-2}> \frac{\frac{n-1}{k+1}-1}{n-k-2}=\frac{1}{k+1}=\frac{|N_{C_1}(w_{k+1})|}{|C_1|}$;
		\item[$\bullet$] $w_i$ for $i>k+1$: it immediately follows from the previous case, since $|N_{C_2}(w_i)|\geq |N_{C_2}(w_{k+1})|$ and $|N_{C_1}(w_i)|= |N_{C_1}(w_{k+1})|$.
	\end{itemize}

	Let us now consider the vertices in~$Q$.
	Vertex $v_1$ is satisfied, since it is universal.
	Next, all vertices in $Q\cap C_2$ that are adjacent to $w_{k+1}$ are also satisfied, since they are adjacent to all the vertices in~$C_2$.
	Consider now vertex~$v_q$.
	From the above, we may assume that it is not adjacent to $w_{k+1}$ (and hence, it has exactly one neighbour in $C_1$, namely $v_1$).
	Notice that, since $w_{k+1}$ has only neighbours in the clique, we have $q\geq d(w_{k+1})> \frac{n-1}{k+1}=\frac{q+s-1}{k+1}$, which implies that $q\cdot k> s-1$.
	Hence, we obtain the following:

	\begin{align*}
		\frac{|N_{C_2}(v_q)|}{|C_2|-1} & \geq \frac{q-1}{q+s-k-2}=\frac{(q-1)(k+1)}{(q+s-k-2)(k+1)}=\frac{qk+q-k-1}{(q+s-k-2)(k+1)} \\
		 & >\frac{s-1+q-k-1}{(q+s-k-2)(k+1)}=\frac{1}{k+1}=\frac{|N_{C_1}(v_q)|}{|C_1|}\,.
	\end{align*}

	Thus, $v_q$ is satisfied with respect to~$\Pi$.
	Finally, all vertices $v_i$, for $i<q$ and which are not adjacent to $w_{k+1}$, are also satisfied, since $|N_{C_2}(v_i)|\geq |N_{C_2}(v_q)|$ and $|N_{C_1}(v_i)|=1$.
	So we may apply the same arguments as for~$v_q$.

	Computing the degrees of $G$ and determining $k=\max\{i:d(w_i)\leq \frac{n-1}{i}\}$ can be done in time $\mathcal{O}(|E|)$ and thus, every connected threshold graph with at least $4$ vertices except stars admits a $2$-community structure that can be found in time $\mathcal{O}(|V|+|E|)$.
	Notice that since $G$ is connected, we have $|E|\geq |V|-1$ and so $\mathcal{O}(|V|+|E|)=\mathcal{O}(|E|)$ in this context.
\end{proof}

Now, let us consider a disconnected threshold graph $G=(V,E)$, i.e.\ a connected threshold graph $H=(V_H,E)$ and a set $I$ of isolated vertices. Assume that $G$ contains at least 4 vertices. If $|I|\geq 2$, say $u,v\in I$, $G$ admits a $2$-community structure $\Pi=\{\{u,v\},V_H\cup (I\setminus\{u,v\})\}$.\newline
Let us now consider the case when $|I|=1$.
In this case, the existence of a $2$-community structure seems less trivial, and there exist infinite families of such graphs that do not admit any $2$-community structure.
We present two of them here.\newline
Let $a,b,s \in \mathbb{N}^{+}$.
The graph $G_{a,b,s}=(Q\cup S \cup \{u\},E)$ (see Figure~\ref{ThNo2comm}) is defined as follows:\\
(i) $Q = \{v_1,\ldots, v_q\}$ is a clique, $S = \{w_1,\ldots,w_s\}$ is a stable set, $u$ is an isolated vertex.\newline
The vertices of the clique are partitioned into two sets: $Q = A \cup B$ with $\vert A \vert = a$ and $\vert B \vert =b$:\newline
(ii) The vertices in $A = \{v_1,\ldots,v_a\}$ are universal vertices in the graph $G\setminus \{u\}$.\\
(iii) The vertices $B=\{v_{a+1},\ldots,v_q\}$ have no neighbour in~$S$.\\
Notice, that $B$ could be empty.
\begin{figure}[!ht]
	\centering
	\includegraphics[width=0.6\linewidth]{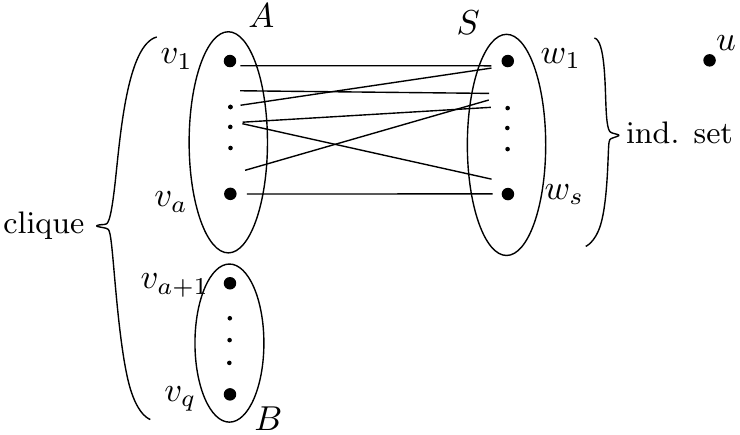}
	\caption{The graph $G_{a,b,s}$.}
	\label{ThNo2comm}
\end{figure}
\begin{theorem}\label{no 2-comm in threshold gr}
	For $a, b, s \in \mathbb{N}^{+}$, such that $b <  \frac{a+s}{s}$, $G_{a,b,s}=(Q\cup S \cup \{u\},E)$ does not admit any $2$-community structure.
\end{theorem}
\begin{proof}
	Assume that $G_{a,b,s}$ admits a $2$-community structure $\Pi=\{C_1,C_2\}$.
	First, note that since the vertices in $B$ are true twins with at least one non-neighbour ($u$), and by Property~\ref{prop-false} we conclude that they must belong to the same community.
	Moreover, if $u \in C_i$ for some $i\in{1,2}$, Property~\ref{allCommunityInneighbourhood} implies that $v_1,\ldots, v_a \in C_{3-i}$.
	Without loss of generality, let $v_1,\ldots, v_a \in C_1$ and $u\in C_2$.
	Next, \Cref{allneighboursSameCommunity} implies that $w_1,\ldots,w_s \in C_1$.
	Finally, since by definition we need to have that $\vert C_2\vert \ge 2$, we conclude that all vertices in $B$ belong to~$C_2$.\\
	The vertices in sets $A$ and $S$ as well as vertex $u$ are trivially satisfied with respect to the partition $\{C_1,C_2\}$ constructed above.
	However, in order for vertices in $B$ to satisfy (\ref{prop:k-comm2}) we must have:
	\begin{equation}
		\label{ineq: vertices in B}
		\frac{b-1}{b}\ge \frac{a}{a+s}\,.
	\end{equation}
	Since by the assumption $b < \frac{a+s}{s}$, \eqref{ineq: vertices in B} is not satisfied by the vertices in~$B$.
	Hence, the graph $G_{a,b,s}$ with $b < \frac{a+s}{s}$ does not admit any $2$-community structure.
\end{proof}

We leave it as an open problem to characterize those disconnected threshold graphs with exactly one isolated vertex that do admit a $2$-community structure.
Also notice that for $k\geq 3$, finding a (generalized) $k$-community structure in threshold graphs is an open problem as well.

\section{Graphs without generalized 2-community structures}
\label{sec-family}

In this section, we introduce an infinite family of connected graphs that do not admit any generalized $2$-community structure.
Notice that in~\cite{bazgan:firstInfiniteFamily}, the authors present a first infinite family of graphs that do not admit any generalized $2$-community structure.
However, the graphs of that infinite family all contain an even number of vertices, while the infinite family introduced here contains both graphs with an even and graphs with an odd number of vertices.\newline
Let $p,l \in \mathbb{N}^{+}$.
The graph $G_{p,l}=(V,E)$ (see \Cref{Gpl}) is defined as follows:\newline
(i) $V=\{u,v_0,\ldots,v_4\}\cup T\cup F$, where $T$ is a set of $p$ vertices $t_1,\ldots,t_p$ which are pairwise true twins and $F$ is a set of $\ell$ vertices $f_1,\ldots,f_\ell$ that are pairwise false twins;\newline
(ii) $u$ is adjacent to all vertices in $V\setminus \{v_0\}$;\newline
(iii)  $E$ contains in addition the edges $v_0v_1,v_1v_2,v_1v_3, v_2v_4$ as well as either the edge $v_2v_3$ or the edge $v_3v_4$;\newline
(iv) finally, $T$ is complete to $F\cup \{u,v_1,v_3\}$.

\begin{figure}[!ht]
	\centering
	\includegraphics[width=0.6\linewidth]{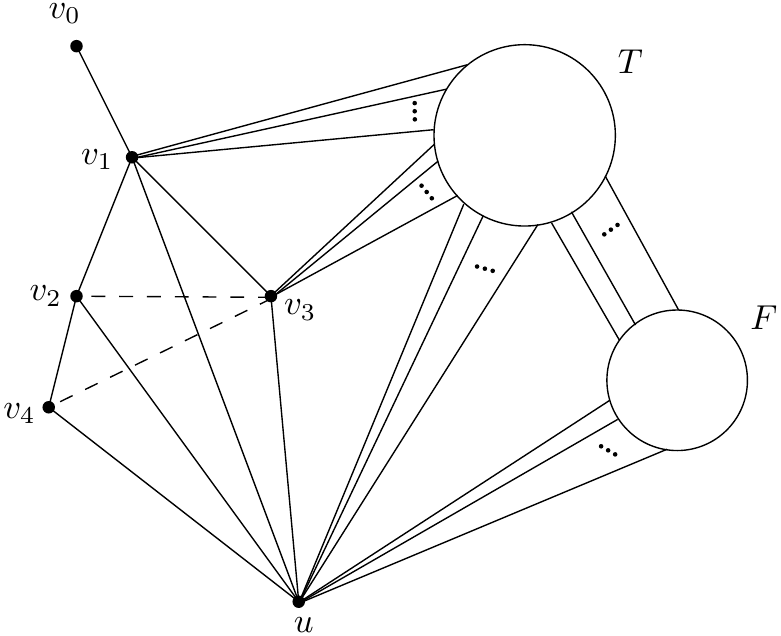}\caption{The graph $G_{p,l}$.
		Exactly one of the dotted edges exists.}
	\label{Gpl}
\end{figure}

Based on this description, we define the infinite family of graphs $\mathcal{G}= \{G_{p,\lceil \frac{p}{2}\rceil}:\ p\geq 3\}$, for which no generalized $2$-community structure exists.
Note that a graph obtained from $G_{p,\lceil \frac{p}{2}\rceil}$ by including both edges $v_2v_3$ and $v_3v_4$ admits a $2$-community structure $\Pi =\{\{v_0,v_1,v_2,v_3,v_4\},T\cup F\cup \{u\}\}$ and a graph obtained from $G_{p,\lceil \frac{p}{2}\rceil}$ by including neither edge $v_2v_3$ nor $v_3v_4$ admits a $2$-community structure  $\Pi =\{\{v_2,v_4,u\} ,T\cup F\cup \{v_0,v_1,v_4\}\}$.

\begin{theorem}
	\label{thm:No2-comm}
	For $p\ge3$, $G_{p,\lceil \frac{p}{2}\rceil} $ does not admit any generalized $2$-community structure.
\end{theorem}

\begin{proof}
	Let $\ell=\lceil \frac{p}{2}\rceil$.
	It follows from \Cref{2commUniv} that $G_{p,\ell}$ does not admit a generalized $2$-community structure  $\Pi=\{C_1,C_2\}$ such that $\vert C_i \vert =1$, for some $i\in \intInterval{2}$.
	Hence, assume by contradiction that $G_{p,\ell}$ admits a $2$-community structure $\Pi=\{C_1,C_2\}$, i.e., such that $\vert C_i\vert \ge 2$, for all $i\in{1,2}$. Without loss of generality, we may assume that $v_0\in C_1$. Then $v_1$ must be in $C_1$ as well, otherwise $N_{C_1}(v_0) = \emptyset$, a contradiction with \Cref{allneighboursSameCommunity}.
	Moreover, by \Cref{allCommunityInneighbourhood}, we have $u\in C_2$.

	Let us consider now the set of true twins~$T$.
	We know from \Cref{prop-false} that they all belong to the same community.
	We distinguish two cases.

	\begin{itemize}
		\item First, assume that $t_i \in C_2$ for all $i\in\intInterval{p}$.
		      Then for all $j\in \intInterval{\ell}$, $N(f_j)=T\cup \{u\}$ implies that $N_{C_1}(f_j)=\emptyset$.
		      It follows from \Cref{allneighboursSameCommunity} that $f_j\in C_2$ for all $j\in \intInterval{\ell}$.
		      On \Cref{Gpl_case1}, we illustrate the current assignment of vertices to communities.

		      \begin{figure}[!ht]
			      \centering
			      \includegraphics[width=0.60\linewidth]{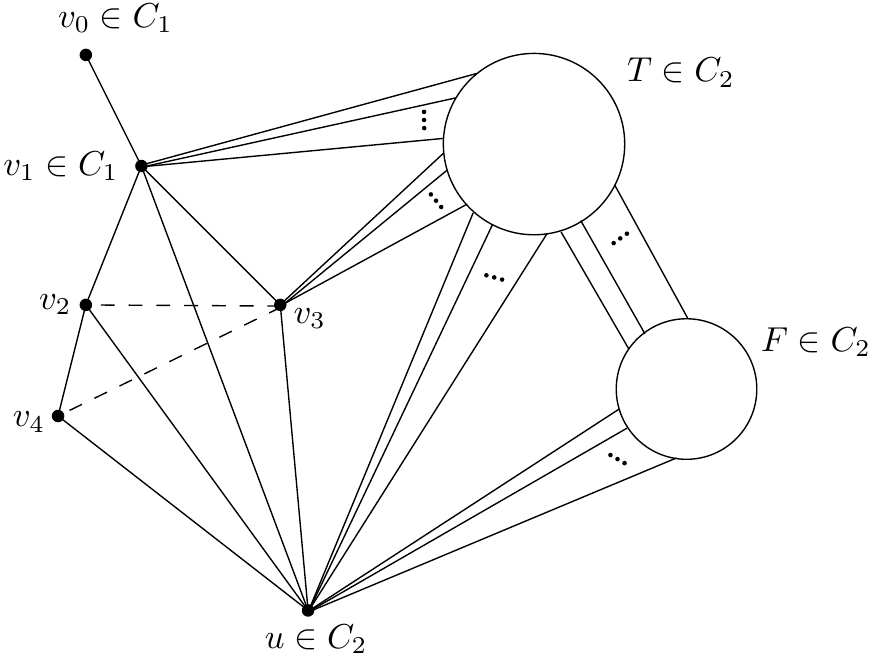}
			      \caption{The graph $G_{p,l}$ and current the assignment of vertices to communities.}
			      \label{Gpl_case1}
		      \end{figure}
		      Let us test $v_2$ on~$C_2$.
		      If $v_2v_3 \in E$, then $\displaystyle\frac{\vert N_{C_2}(v_2)\vert}{\vert C_2 \vert -1}=\displaystyle\frac{3}{p+\ell+3}$ and $\displaystyle\frac{\vert N_{C_1}(v_2)\vert}{\vert C_1\vert}=\frac{1}{2}$.
		      If $v_4v_3\in E$, then $\displaystyle\frac{\vert N_{C_2}(v_2)\vert}{\vert C_2 \vert -1}=\displaystyle\frac{2}{p+\ell+2}$ and $\displaystyle\frac{\vert N_{C_1}(v_2)\vert}{\vert C_1\vert}=\frac{1}{3}$.
		      Since $p+\ell\ge 5$, \eqref{prop:k-comm2} fails for $v_2$ in both cases, and we conclude from \Cref{prop:preassignment} that $v_2 \in C_1$.
		      Further, let us now test $v_4$ on~$C_2$.
		      If $v_2v_3 \in E$, then we have that $\displaystyle\frac{\vert N_{C_2}(v_4)\vert}{\vert C_2 \vert -1}=\displaystyle\frac{1}{p+\ell+1}$ and $\displaystyle\frac{\vert N_{C_1}(v_4)\vert}{\vert C_1 \vert}= \frac{1}{4}$ if $v_2v_3\in E$, respectively $\displaystyle\frac{\vert N_{C_2}(v_4)\vert}{\vert C_2 \vert -1}=\displaystyle\frac{2}{p+\ell+2}$ and $\displaystyle\frac{\vert N_{C_1}(v_4)\vert}{\vert C_1 \vert}= \frac{1}{3}$  if $v_3v_4\in E$.
		      Since by assumption $p+\ell\ge 5$,  \eqref{prop:k-comm2} fails for $v_4$, and we conclude that $v_4\in C_1$.

		      Let us consider vertex~$v_3$.
		      Assume that $v_3 \in C_1$.
		      Then $\displaystyle \frac{\vert N_{C_1}(v_3)\vert}{\vert C_1 \vert -1}=\displaystyle\frac{1}{2}$ and $\displaystyle \frac{\vert N_{C_2}(v_3)\vert}{\vert C_2 \vert}=\displaystyle\frac{1+p}{1+p+\ell}> \frac{1+p}{1+p +\frac{p}{2} +1} > \frac{1+p}{2+2p} = \frac{1}{2}$, a contradiction since \eqref{prop:k-comm2} fails for~$v_3$.
		      So $v_3 \in C_2$.
		      But then $\displaystyle \frac{\vert N_{C_1}(v_1)\vert}{\vert C_1 \vert -1}=\displaystyle\frac{2}{3}$ and $\displaystyle \frac{\vert N_{C_2}(v_1)\vert}{\vert C_2 \vert}=\frac{p+2}{p+\ell+2}>\frac{p+2}{p+\frac{p}{2}+1+2}=\frac{2}{3}$, a contradiction since \eqref{prop:k-comm2} fails for~$v_1$.

		      We conclude that $t_i\not\in C_2$, for all $i\in\intInterval{p}$.

		\item Assume that $t_i \in C_1$ for all $i\in\intInterval{p}$.
		      On \Cref{Gpl_case2}, we illustrate the current assignment of vertices to communities.

		      \begin{figure}[!ht]
			      \centering
			      \includegraphics[width=0.55\linewidth]{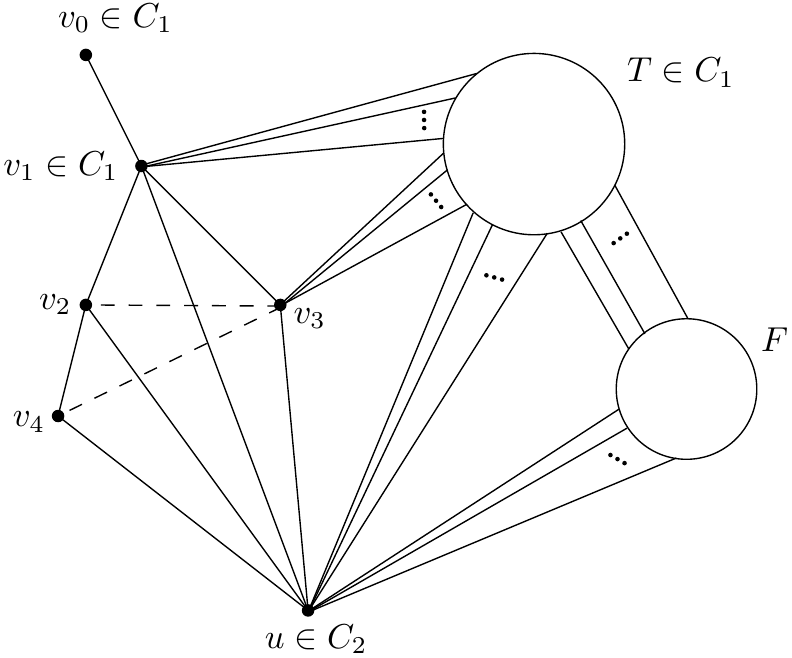}
			      \caption{The graph $G_{p,l}$ and the current assignment of vertices to  communities.}
			      \label{Gpl_case2}
		      \end{figure}

		      Let us test $v_4$ on~$C_1$.
		      Then, if $v_2v_3\in E$, we have  $\displaystyle\frac{\vert N_{C_1}(v_4)\vert}{\vert C_1 \vert -1}=\displaystyle\frac{1}{3+p}$ and $\displaystyle\frac{\vert N_{C_2}(v_4)\vert}{\vert C_2 \vert}=\frac{1}{2+\ell}$.
		      Since by assumption we have $p >\ell$, \eqref{prop:k-comm2} fails.
		      Similarly, if $v_3v_4 \in E$, we have $\displaystyle\frac{\vert N_{C_1}(v_4)\vert}{\vert C_1 \vert -1}=\displaystyle\frac{2}{4+p}$ and $\displaystyle\frac{\vert N_{C_2}(v_4)\vert}{\vert C_2 \vert} = \frac{1}{1+\ell}>\frac{1}{1+\frac{1}{2}(p+2)} = \frac{2}{p+4}$, a contradiction since \eqref{prop:k-comm2} fails.
		      Hence, we conclude from \Cref{prop:preassignment} that $v_4 \in C_2$ in both cases.
		      Further, let us now test $v_2$ on~$C_1$.
		      Then, if $v_2v_3\in E$, $\displaystyle\frac{\vert N_{C_1}(v_2)\vert}{\vert C_1 \vert -1}=\displaystyle\frac{2}{3 + p} $ and $\displaystyle\frac{\vert N_{C_2}(v_2)\vert}{\vert C_2 \vert}= \frac{2}{2+\ell}$.
		      Since by assumption we have $p > \ell$, \eqref{prop:k-comm2} fails for~$v_2$.
		      If $v_3v_4\in E$, $\displaystyle\frac{\vert N_{C_1}(v_2)\vert}{\vert C_1 \vert -1}=\displaystyle\frac{1}{2+p}$ and  $\displaystyle\frac{\vert N_{C_2}(v_2)\vert}{\vert C_2 \vert}= \frac{2}{3+\ell} > \frac{2}{3+\frac{1}{2}(p+2)} = \frac{4}{8+p}$.
		      Hence, in order for \eqref{prop:k-comm2} to hold, we must have $p\geq4p$, a contradiction since $p\geq 3$.
		      Thus, we conclude from \Cref{prop:preassignment} that $v_2 \in C_2$ in both cases.

		      Further, let us consider the vertices in~$F$.
		      We show that $f_i\in C_1$ for all $i\in\intInterval{\ell}$.
		      Let us test $f_i$ on $C_2$ (recall that F is a set of false twins, which are not adjacent to each other, and therefore do not necessarily belong to the same community), for any $i\in \intInterval{\ell}$.
		      Then, independently whether $v_2v_3\in E$ or $v_3v_4\in E$, $\displaystyle \frac{\vert N_{C_2}(f_i)\vert }{\vert C_2\vert -1}=\displaystyle\frac{1}{3}$ and $\displaystyle \frac{\vert N_{C_1}(f_i)\vert }{\vert C_1\vert} = \frac{p}{p+\ell+2}$.
		      Since by assumption $p \ge 3$ and $p>\ell$, we get a contradiction because \eqref{prop:k-comm2} fails for~$f_i$.
		      Hence, we conclude from \Cref{prop:preassignment} that $f_i \in C_1$ for all $i\in\intInterval{\ell}$.

		      Finally, let us consider $v_3$ and test it on~$C_1$.
		      Then, independently whether $v_2v_3\in E$ or $v_3v_4\in E$,  $\displaystyle \frac{\vert N_{C_1}(v_3)\vert }{\vert C_1\vert -1}=\displaystyle\frac{1+p}{2+p+\ell}$ and $\displaystyle \frac{\vert N_{C_2}(v_3)\vert }{\vert C_2\vert} =\frac{2}{3}$.
		      Since by assumption we have $\frac{1}{2}(p-1)<\ell$, we get that $\displaystyle\frac{1+p}{2+p+\ell} < \frac{1+p}{2+p+\frac{1}{2}(p-1)} = \frac{2(1+p)}{3(1+p)} = \frac{2}{3}$, a contradiction since \eqref{prop:k-comm2} fails for~$v_3$.
		      Hence, it follows from \Cref{prop:preassignment} that $v_3 \in C_2$.
		      Then, independently whether $v_2v_3\in E$ or $v_3v_4\in E$, in order to satisfy \eqref{prop:k-comm2} for $v_1$, we must have $\displaystyle\frac{\vert N_{C_1}(v_1)\vert}{\vert C_1 \vert -1}=\displaystyle\frac{p+1}{p+\ell+1}\ge \displaystyle\frac{3}{4}=\displaystyle\frac{\vert N_{C_2}(v_1)\vert}{\vert C_2 \vert}$.
		      But, by the argument above $\displaystyle\frac{p+1}{p+\ell+1}\leq\frac{p+1}{p+\frac{p}{2}+1}=\frac{2p+2}{3p+2}< \frac{3}{4}$.
		      Hence, we get a contradiction since \eqref{prop:k-comm2} fails for $v_1$, and we therefore cannot have $f_i \in C_1$ for all $i\in \{1,\ldots,p\}$.
	\end{itemize}

	We conclude from the above that $G_{p,\lceil \frac{p}{2}\rceil}$ does not admit any $2$-community structure.
\end{proof}

Note that one can easily define an integer linear program (ILP) in order to check whether a graph $G$ admits a (generalized) $2$-community structure or not.
Using this approach, we could test all connected non-isomorphic graphs up to 11 vertices for the existence of a $2$-community structure.
The complete list of these graphs was obtained by using the algorithm developed in~\cite{nauty}.
Our main findings are that, excluding the stars (that can be discarded due to \Cref{k-matching}):

\begin{itemize}
	\item all connected graphs from 4 to 9 vertices admit a $2$-community structure;
	\item with 10 vertices, only 4 connected graphs do not admit a $2$-community structure; these graphs all belong to the family presented in~\cite{bazgan:firstInfiniteFamily};
	\item with 11 vertices, there are only 6 connected graphs not admitting any $2$-community structure;
	\item with 12 vertices, there are many (more than 100) connected graphs not admitting any $2$-community structure.
\end{itemize}

These findings may help in order to better understand the structure of those graphs that do not admit any $2$-community structure.

\section{Conclusion}
\label{sec-conclusion}

In this paper, we investigated (generalized) $k$-community structures and gave new results for forests and (connected) threshold graphs.
We also presented a first infinite family of graphs that do not admit any generalized $2$-community structure and such that the graphs may contain an even or an odd number of vertices.
There remain several interesting open questions, some of which we present hereafter.

\begin{itemize}
	\item What is the complexity of deciding whether a given graph admits a $2$-community structure?
	\item Which disconnected threshold graphs admit a $2$-community structure?
	\item Which (connected) threshold graphs admit a $k$-community structure, for $k\geq 3$?
	\item Can we extend our results on (generalized) $2$-community structures to larger graph classes? A natural extension would be split graphs, eventually leading towards chordal graphs (which generalize both split graphs and forests).
\end{itemize}

\bibliographystyle{elsarticle-num} 
\bibliography{sn-bibliography.bib}

\end{document}